\documentclass[1p]{elsarticle}
\usepackage[all]{xy}
\usepackage{graphicx,amsmath,amssymb,amsthm,amscd,mathrsfs}
\newcommand{\N}{\mathbb{N}}
\newcommand{\Z}{\mathbb{Z}}

\newcommand{\R}{\mathbb{R}}
\newcommand{\C}{\mathbb{C}}

\newcommand{\Ad}{\mathrm{Ad}}

\DeclareMathOperator{\supp}{supp}

\DeclareMathOperator{\Ker}{Ker}

\newtheorem{thm}{Theorem}[section]
\newtheorem{df}[thm]{Definition}
\newtheorem{prop}[thm]{Proposition}
\newtheorem{lem}[thm]{Lemma}
\newtheorem{cor}[thm]{Corollary}
\newtheorem{rem}[thm]{Remark}
\newtheorem{ass}[thm]{Assumption}

\begin{document}
\begin{frontmatter}
\title{Linear differential equations on $\mathbb{P}^{1}$ and root systems}
\date{}
\author{Kazuki Hiroe}
\ead{kazuki@kurims.kyoto-u.ac.jp}
\address{Research Institute of Mathematical Sciences, 
    Kyoto University, Kyoto 606-8502 Japan.
}

\begin{abstract}
    In this paper, we study the Euler transform on  
    linear ordinary differential operators on $\mathbb{P}^{1}$.
    The spectral type is the tuple of integers which 
    count the multiplicities of 
    local formal  solutions with the same leading terms.
    We compute the changes of spectral types under the 
    action of the Euler transform and show that the 
    changes of spectral types generate a transformation
    group of  
    a $\mathbb{Z}$-lattice which is isomorphic to a quotient lattice 
    of a Kac-Moody root lattice with the Weyl group as the transformation
     group.
\end{abstract}
\end{frontmatter}
\section*{Introduction}
The integral transformation
 \[
  I_{a}^{\lambda}f(x)=\frac{1}{\Gamma(\lambda)}
  \int_{a}^{x}(x-t)^{\lambda-1}f(t)\,dt,\quad (\mathrm{Re}\lambda>0)
   \]
   and its analytic continuation with respect to $\lambda\in \mathbb{C}$
   is called the Euler transform (or Riemann-Liouville integral) of $f(x)$ 
for $a,\lambda\in\C$. 
This integration is  fundamental for  the  theory of fractional 
calculus  
because of the following observation.
If $f(x)$ satisfies suitable conditions, for example 
$f(x)$ is holomorphic on a neighbourhood of $x=a$ or 
$f(x)=(x-a)^{\alpha}\phi(x)$ 
where $\mathrm{Re\,}\alpha>-1$ and $\phi(x)$ 
is a holomorphic function on a neighbourhood of $x=a$ and $\phi(a)\neq 0$,
then it is known that
   \[
     I^{-n}_{a}f(x)=\frac{d^{n}}{dx^{n}}f(x).
   \]
 Hence one can regard the  Euler transform as 
 a fractional or complex powers of the derivation $\partial=\frac{d}{dx}$.
 This may allow us to write 
 $\partial^{\lambda}f(x)=I_{a}^{-\lambda}f(x)$ formally. 

 Moreover one can show a generalization of the Leibniz rule, for example,
   \[
   \partial^{\lambda}p(x)\psi(x)=\sum_{i=0}^{n}\begin{pmatrix}\lambda\\i\end{pmatrix}p^{(i)}(x)\partial^{\lambda-i}\phi(x),
   \] 
 where $p(x)$ is a polynomial of degree equal to or less than $n$. 
 Now let us consider a differential operator with polynomial coefficients,
 \[
  P(x,\partial)=\sum_{i=0}^{n}a_{i}(x)\partial^{i}.
 \]
 The above Leibniz rule assures that 
 $Q(\partial,x)=
 \partial^{\lambda+m}P(x,\partial)\partial^{-\lambda}$ is again
 the new differential operator with polynomial coefficients 
 if we choose a suitable $m\in\Z$. 
 Moreover if $f(x)$ satisfies $P(x,\partial)f(x)=0$ 
 and $I^{-\lambda}_{a}f(x)$ is well-defined 
 for some $a,\lambda\in \C$, then we carry out the formal
 computation,
 \begin{align*}
  \partial^{\lambda+m}P(x,\partial)
  \partial^{-\lambda}I_{a}^{-\lambda}f(x)&
  =\partial^{\lambda+m}P(x,\partial)\partial^{-\lambda+\lambda}f(x)\\
 &=\partial^{-\lambda+m}P(x,\partial)f(x)\\
 &=0.
\end{align*}
Thus 
we can obtain the following observation.
 The fractional derivative
 $\partial^{\lambda}$ turns a differential equation 
 with polynomial coefficients $P(x,\partial)u=0$ 
 into a new differential equation 
 with polynomial coefficients $Q(x,\partial)u=0$, 
 and moreover a solution of $Q(x,\partial)u=0$ 
 is given by a solution of $P(x,\partial)u=0$ 
 if the Riemann-Liouville integral is well-defined and 
 satisfies suitable conditions. 
 Thus it is natural to be wondering about what kind of 
 differential  equations can be obtained 
 by the Euler transform 
 from known equations or 
 how we can reduce 
 a difficult equation to an easier one.
                                                              
 Let $K$ be an algebraically closed field of 
 characteristic zero and $W(x)$ the ring of  
 differential operators with coefficients in $K(x)$, 
 the field of rational functions.  
 In \cite{O}, T. Oshima gives 
 an algebraic definition of 
 the Euler transform on $W(x)$ 
 as an analogue of 
 the middle convolution defined by N. Katz 
 in \cite{K}. 
 In this paper, we shall consider a generalization
 of the works of Katz and Oshima who mainly study
 Fucshian differential operators, i.e, operators 
 only with regular singular points.
 Namely we shall follow Oshima's definition of the algebraic Euler 
 transform and investigate the properties of it on the 
 theory of linear differential operators on $\mathbb{P}^{1}$
with irregular singular points. 
 
On the other hand, in \cite{C}, 
W. Crawley-Boevey clarifies
the correspondence between systems of first order Fuchsian 
linear differential equations 
and certain representations of quivers. 
As a consequence of this correspondence,
he gives the necessary and sufficient condition 
of the existence of irreducible differential equations
with the prescribed local data
by using the existence theorem of irreducible representations
of quivers. 
Moreover he shows the middle convolution (or the Euler transform) 
can be obtained from the 
operations on representations of quivers
, so-called the reflection functors which
induce the actions  of the Weyl groups on 
the spaces of dimension vectors of representations.
This picture enable us to realize the changes of local data
of Fuchsian differential equations 
given by the middle convolution 
in terms of  the actions of the Weyl groups
on the root systems of the quivers.

 In this paper, we mainly deal with differential equations 
 with at most unramified irregular singular points 
 and give a realization of the action of the Euler transform 
 in terms of the action of the Weyl group of a Kac-Moody root 
 system as a generalization of Crawley-Boevey's result. 

 Our result can be roughly explained  as follows. 
 Let us take $P\in W(x)$ with at most 
 unramified irregular singular points. 
 We impose some generic conditions on $P$ 
 (see Section \ref{Euler transform and Weyl group} for precise conditions).
 From local structures around singular points, we shall define 
 the notion of the spectral type, the tuple of positive 
 integers which 
 count the multiplicities of local formal 
 solutions with the same leading terms.
We compute the changes of the spectral type given by
the Euler transform and other algebraic transformations 
explicitly.
Then we show that  the changes of the spectral type 
give  $\mathbb{Z}$-lattice automorphisms and 
these automorphisms generates a transformation group on this lattice.
Let us denote by $L(P)$ and $\tilde{W}(P)$ the lattice and the 
transformation group respectively.
 Then finally we shall show that 
 the $\tilde{W}(P)$-module $L(P)$ 
 is isomorphic to a quotient lattice of a 
 Kac-Moody root lattice with the Weyl group action.
 That is to say,  
 there exists the root lattice $Q(P)$ and the Weyl group $W(P)$ 
 associated with a symmetric Kac-Moody root system such that $L(P)$
 is isomorphic to a quotient of $W(P)$-module $Q(P)$
 (see Theorem \ref{Weyl group}). 
      
  Furthermore, 
  we define a generalization of the root system in $L(P)$ 
  as an analogue of the root system of $Q(P)$. 
  Then we show that if $P$ is irreducible, 
  then the spectral type of $P$  
  is the root of this generalized root system 
  (see Theorem \ref{root and irreducible}).

 As a corollary, 
 we can show an analogue of the Katz algorithm of the differential
 operators with irregular singular points
 obtained by D. Arinkin and D. Yamakawa 
 independently (\cite{A}, \cite{Y}). 
 
 In \cite{Bo}, P. Boalch considers vector bundles with 
 meromorphic connections 
 on $\mathbb{P}^{1}$
 which have finitely many regular singular points 
 and one unramified irregular singular point. 
 He gives a correspondence between these connections 
 and representations of quivers as a generalization 
 of the result of Crawley-Boevey. 
 If we restrict our case to Boalch's setting, 
 we can obtain the root systems
 whose Dynkin diagrams agree with Boalch's quivers if we forget the 
 orientations of quivers.

 As examples of our correspondence with root systems, 
 let us consider confluent equations of Heun's differential equations. 
 Then we can obtain extended Dynkin diagrams of affine Lie algebras, 
 $D^{(1)}_{4}$, $A^{(1)}_{3}$, $A^{(1)}_{2}$, $A^{(1)}_{1}$ 
 and $A^{(1)}_{1}\oplus A^{(1)}_{1}$. 
 These agree with symmetries of B\"acklund 
 transforms of Painlev\'e equations 
 which are obtained from these Heun's equations 
 with an apparent singular point (see Section \ref{Heun and Painleve}).
	 
\section{Local structures of differential operators}
Let us give a review of a formal theory of differential operators around 
regular and irregular singular points. 
The contents of this section are well-known 
and found in standard references 
(for example \cite{M}, \cite{M1}, \cite{S}, \cite{W}, etc.).

\subsection{Basic notation}
Let us denote by  $K$ 
an algebraically closed field of characteristic zero. 
We denote by $K[x]$, $K(x)$ and  $K( (x))$ the ring of polynomials, 
the field of rational functions and the quotient field of the ring of 
formal power series $K[ [x] ]$ respectively.
Since $K$ is algebraically closed, any element $f(x)\in K(x)$ 
decomposes as a product of linear factors, $f(x)=\prod_{i=1}^{m}(x-a_{i})^{m_{i}}$ where 
$m_{i}\in \mathbb{Z}$ and $a_{i}\in K$ 
satisfying $a_{i}\neq a_{j}$ if $i\neq j$.
If $m_{i}>0$ (resp. $m_{i}<0$), then $a_{i}$ is called 
the \textit{zero} (resp. \textit{pole}) of $f$.
We can define the discrete valuation $v_{a}$ for any $a\in K\cup\{\infty\}$
on $K(x)$ as follows.
As we see, for any $f(x)$ and $a\in K$ there exists $m_{a}\in \mathbb{Z}$
such that $f(x)=(x-a)^{m_{a}}g(x)$ where $g(x)$ has
no pole and no zero at $a\in K$.
Then we define $v_{a}(f(x))=m_{a}$.
If $f(x)=0$, we define $v_{a}(0)=\infty$.
Similarly we define $v_{\infty}(f(x))=v_{0}(f(x^{-1}))$ for any 
$f(x)\in K(x)$.
We have natural embeddings $K(x)\hookrightarrow K( (x-a))$ 
for any $a\in K$ 
and $K(x)\hookrightarrow K( (x^{-1}))$
which are obtained by the completions of $K(x)$ with respect
to the valuations $v_{a}$ and $v_{\infty}$ respectively
(see \cite{Iw} for example).

Let  $\partial =\frac{d}{dx}$ be the natural differential operator 
on these rings.
We denote by $W[x]$, $W(x)$ and $\widehat{W}(x)$, 
the ring of differential operators 
with coefficients in $K[x]$, $K(x)$ and $K( (x))$ respectively.
Namely,
$W[x]=K[x][\partial]$, $W(x)=K(x)[\partial]$ and 
$\widehat{W}(x)=K( (x))[\partial]$.

Let $\mathcal{F}$ be one of the rings of differential operators defined above. 
The rank of $P=\sum_{i\ge 0}a_{i}(x)\partial^{i}\in \mathcal{F}$ 
is $\max\{i\mid a_{i}(x)\neq 0\}$ 
and denoted by $\mathrm{rank\,}P$.  
If in particular $\mathcal{F}=W[x]$, 
we define the degree by $\deg P=\max\{\deg a_{i}(x)\mid i=0,1,\ldots\}$.

We associate $P\in \mathcal{F}$ with 
the left $\mathcal{F}$-module $\mathcal{F}/\langle P\rangle $ where $\langle P\rangle$ is 
the left ideal of $\mathcal{F}$ generated by $P$.
If $\mathcal{F}$ is  $W(x)$ or  $\widehat{W}(x)$,
whose coefficient set  is  the field, 
then we can regard $\mathcal{F}/\langle P\rangle$ as
the finite dimensional vector space
over the coefficient field of $\mathcal{F}$ 
with $\dim \mathcal{F}/\langle P\rangle=\mathrm{rank\,}P$.
We denote $\mathcal{F}/\langle P\rangle$ by $M_{P}(x)$ 
and $\widehat{M}_{P}(x)$
if $\mathcal{F}$ is $W(x)$ and  $\widehat{W}(x)$ respectively.

For any $a\in K$ and $\infty$, we define algebra homomorphisms
\begin{align*}
    &\begin{array}{cccc}
        \phi_{a}\colon&K( (x))&\longrightarrow&K( (z))\\
        &x&\longmapsto &z+a
    \end{array},\\
    &\begin{array}{cccc}
        \phi_{\infty}\colon&K(x)&\longrightarrow&K( (z))\\
        &x&\longmapsto &z^{-1}
    \end{array}.
\end{align*}
These extend to 
\begin{align*}
    &\begin{array}{cccc}
        \phi_{a}\colon&\widehat{W}(x)&\longrightarrow&\widehat{W}(z)\\
        &x&\longmapsto &z+a\\
        &\partial_{x}=\frac{d}{dx}&\longmapsto&\partial_{z}=\frac{d}{dz}
    \end{array},\\
    &\begin{array}{cccc}
        \phi_{\infty}\colon&W(x)&
        \longrightarrow&\widehat{W}(z)\\
        &x&\longmapsto &z^{-1}\\
        &\partial_{x}&\longmapsto&-z^{2}\partial_{z}
    \end{array}.
\end{align*}
Let us fix an algebraic field extension $K( (t))$ of $K( (x))$ 
where $t^{q}=x$ $(q\in \mathbb{Z}_{>0})$.
Then the natural embedding
\[
    \begin{array}{cccc}
        r_{q}\colon &K( (x))&\longrightarrow&K( (t))\\
        &x&\longmapsto&t^{q}
    \end{array}\quad
\]
extends to
\[
    \begin{array}{cccc}
        r_{q}\colon &\widehat{W}(x)&\longrightarrow&\widehat{W}(t)\\
        &x&\longmapsto&t^{q}\\
        &\partial_{x}&\longmapsto&\frac{1}{q}t^{1-q}\partial_{t}
    \end{array}.
\]
\subsection{Singular points}
The \textit{valuation} $v$ of $K( (x))$ is defined by
\[
    v\left(\sum_{i=-\infty}^{\infty}c_{i}(x)^{m}\right)=\min\{i\mid c_{i}\neq 0\}.
\]
Here we define $v(0)=\infty$. 
We can extend $v$ to $\widehat{W}(x)$.
For $P=\sum_{i\ge 0}a_{i}(x)\partial^{i}\in \widehat{W}(x)$, 
\[
    v(P)=\min_{i\ge 0}\{v(a_{i}(x))-i\}.
\]
Set $P=\sum_{i\ge 0}\sum_{j>-\infty}c_{i,j}x^{j}\partial^{i}\in 
\widehat{W}(x)\backslash\{0\}$ 
and put $v(P)=m$. 
Then the \textit{characteristic polynomial} $ch(P)$ of $P$ is
\[
    ch(P)(s)=\sum_{i\ge 0}c_{i,m+i}s(s-1)\cdots(s-i+1).
\]
Roots of $ch(P)(s)=0$ are called \textit{characteristic exponents} of $P$. 
In particular, if $P$ satisfies 
$\deg_{K[s]}ch(P)(s)=\mathrm{rank\,}P$, 
we say $P$ is \textit{regular singular type}.

Now recall the definition of singular points of elements in $W(x)$.
For $a\in K$ and $\infty$ we can consider embeddings
$W(x)\hookrightarrow \widehat{W}(x)
\stackrel{\phi_{a}}{\rightarrow}\widehat{W}(z)$ and
$W(x)\stackrel{\phi_{\infty}}{\hookrightarrow}\widehat{W}(z)$ respectively.
We also denote these embeddings by 
the same notation 
$\phi_{a},\phi_{\infty}\colon W(x)\hookrightarrow \widehat{W}(z)$.
The \textit{singular points} of $P=\sum_{i=0}^{n}a_{i}(x)\partial^{i}\in W(x)\ (a_{n}(x)\neq 0)$ are
poles of $\frac{a_{i}(x)}{a_{n}(x)}$ for $i=0,\ldots,n-1$. 
Moreover if $z=0$ is a singular point of 
$\phi_{\infty}(P)\in W(z)\subset \widehat{W}(z)$,
then we say $P$ has a singular point at $\infty$.
Let us denote the set of singular points of $P$ 
by $S_{P}\subset K\cup \{\infty\}$.
\textit{Characteristic polynomials} and \textit{characteristic exponents} 
at singular points are defined by
\[
    ch_{a}(P)(s)=ch(\phi_{a}(P))(s) \quad (a\in S_{P})
\]
and their zeros respectively.

\begin{df}[regular singular points and irregular singular points]
    \normalfont
    Consider $P\in W(x)$ with $\mathrm{rank\,}P\ge 1$. 
    A singular point $a\in S_{P}$ is called a {\em regular singular point}
    if $\deg_{K[s]}ch_{a}(P)(s)=\mathrm{rank\,}(P).$
    Otherwise $a\in S_{P}$ is called an \textsl{irregular singular point}.
\end{df}

\subsection{Decomposition of differential operators}
Let us recall the decomposition of differential operators 
studied by Hukuhara \cite{Hu}, Turrittin \cite{T}, 
Malgrange \cite{M}, Robba \cite{R} and the other researchers.

Let us take $P=\sum_{i=0}^{n}a_{i}(x)\partial^{i}\in \widehat{W}(x)$ 
and $w\in K(x)$,
and define 
\[
\mathrm{Ad}(e^{w})P=
\sum_{i=0}^{n}a_{i}(x)\left(\partial-\frac{d}{dx}w\right)^{i}\in \widehat{W}(x).
\]
Sometimes we denote $P^{w}=\mathrm{Ad}(e^{w})P$ for short.
Moreover define $\mathrm{Ad}_{a}(e^{w})
=\phi^{-1}_{a}\circ \mathrm{Ad}(e^{w})\circ \phi_{a}
\colon W(x)\rightarrow W(x)$ for $a\in K\cup \{\infty\}$.
\begin{df}[elementary components]
    \normalfont
    Let $K( (t))$ be an algebraic extension of $K( (x))$ with $t^{q}=x$.
    Let $r_{q}\colon \widehat{W}(x)\hookrightarrow \widehat{W}(t)$ be
    the extension of the natural embedding 
    $K( (x))\hookrightarrow K( (t))$.
    For $P\in \widehat{W}(x)$ we suppose that 
    there exist $w\in t^{-1}K[t^{-1}]$
    and $Q\in \widehat{W}(t)$ of regular singular type 
    with $\mathrm{rank\,}Q\ge 1$ such that
    \begin{enumerate}
        \item $Q$ is monic, i.e., 
            $Q=\partial^{n}+a_{n-1}(x)\partial^{n-1}+\cdots +a_{0}(x)$,
            \,$a_{i}(x)\in K( (t))$ for $i=0,\ldots,n-1$,
        \item $r_{q}(P)=RQ^{w}$ for some $R\in \widehat{W}(t)$,
        \item $ch(r_{q}(P)^{-w})=ch(Q)$.
    \end{enumerate}
    Then we call the pair $(w,Q)$ an {\em elementary component} 
    of $P$ or {\em e-component} shortly and 
    the integer $q$ is called the {\em ramification index} 
    of the elementary
    component $(w,Q)$.
    In particular, if the ramification index of $(w,Q)$ is $q=1$, 
    then we say $(w,Q)$ is {\em unramified}.
\end{df}
\begin{rem}
    \normalfont
    We retain the above notation.
    If $Q'\in \widehat{W}( (t))$ satisfies conditions 2 and 3 in
    the above definition. 
    Then there exists $f\in K( (t))$ such that $Q'=fQ$ 
    $($see th\'eor\`eme 2.4 in \cite{R}$)$.
\end{rem}
\begin{df}[slopes of e-components]
    \normalfont
    Let us consider $P\in \widehat{W}(x)$ and an e-component 
    $(w,Q)$ of $P$ in $\widehat{W}(t)$, an extension of $\widehat{W}(x)$
    with $t^{q}=x$.
    Then the {\em slope} of $(w,Q)$ is defined by 
    $\lambda=\frac{\delta(w)}{q}$
    where $\delta(w)=\deg_{K[t^{-1}]}w$.
\end{df}

The following is one of the most fundamental theorem for 
the formal theory of differential operators in $\widehat{W}(x)$.
\begin{thm}[see \cite{M} and \cite{R} for example]\label{HTL}
    Consider $P\in \widehat{W}(x)$. 
    Then there exists an algebraic extension $K( (t))$ of $K( (x))$ 
    with $t^{q}=x$ such that 
    $r_{q}(P)$ decomposes as follows.

    There exist the unique set of elementary components 
    $\{(w_{1},P_{1}),\ldots,(w_{r},P_{r})\}$ of $P$ 
    in $\widehat{W}(t)$ such that  
    \begin{enumerate}
        \item $w_{i}\neq w_{j}$ if $i\neq j$,
        \item we have the $\widehat{W}(t)$-module decomposition
        $\widehat{M}_{r_{q}(P)}(t)=
        \oplus_{i=1}^{r}\widehat{M}_{P_{i}^{w_{i}}}(t).$
    \end{enumerate}
    We call $\{(w_{1},P_{1}),\ldots,(w_{r},P_{r})\}$  
    the {\em complete set of  e-components} of $P$. 
\end{thm}
If we can choose $q=1$ in the above theorem,
we say $P\in \widehat{W}(x)$ is \textsl{unramified}.

Let us recall of the Newton polygon of 
$P=\sum_{i=0}^{n}a_{i}(x)\partial^{i}\in\widehat{W}(x)$ 
introduced by Malgrange \cite{M1} and Ramis \cite{Ra}. 
Let us associate the points 
\[
(i,v(a_{i}(x))-i)\in \N\times\Z
\]
for $i$-th terms $a_{i}(x)\partial^{i}$ of $P$. 
Then the convex hull of the set
\[
    \bigcup_{i=0}^{n}\{(i-s,v(a_{i}(x))-i+t)\mid  s,t\in \R_{\ge 0}\}
    \subset \R^{2} 
\]
is called the \textit{Newton polygon} of $P$ and written by $N(P)$.

Let us see a relation between $N(P)$ and e-components of $r_{q}(P)$.
We use the same notation as in Theorem \ref{HTL}.

Let 
\[
a_{1}=(i_{1},j_{1}),\ldots,a_{l}=(i_{l},j_{l})\quad 
(0\le i_{1}<\cdots <i_{l})
\]
be the vertices of $N(P)$.
Also let  $\lambda_{k}$ be slopes of edges from $a_{k}$ 
to $a_{k+1}$ for $k=1,\ldots,l-1$, i.e.,
$\lambda_{k}=\frac{j_{k+1}-j_{k}}{i_{k+1}-i_{k}}$.
Then we can see that $\{\lambda_{k}\mid k=1,\ldots,l-1\}=
\left\{\frac{\delta(w_{i})}{q}\mid i=1,\ldots,
r\right\}$
 and 
\begin{equation}\label{NP finite}
    \begin{aligned}
i_{1}&=\begin{cases}
    0&\text{if $w_{i}\neq 0$ for all $i=1,\ldots,r$},\\
    \mathrm{rank\,} P_{\tilde{i}}&\text{if there exists }w_{\tilde{i}}=0,
\end{cases}\\
i_{k+1}-i_{k}&
=\sum_{\left\{i\,\Big|\, \frac{\delta(w_{i})}{q}=\lambda_{k}\right\}}
\mathrm{rank\,}P_{i}\quad (k=1,\ldots,l-1),\\
j_{1}&=v(P),\\
j_{k+1}-j_{k}&=\lambda_{k}\sum_{\left\{i\,\Big|\, \frac{\delta(w_{i})}{q}=
\lambda_{k}\right\}}\mathrm{rank\,}P_{i}\quad (k=1,\ldots,l-1),\\
j_{l}&=v(a_{n}(x))-n.
\end{aligned}
\end{equation}

Now we can define Newton polygons of $P\in W(x)$ at singular points. 
Namely for $a\in S_{P}$, $N_{a}(P)=N(\phi_{a}(P))$ denote 
the Newton polygon of 
$\phi_{a}(P)$.

In particular let us see $N_{\infty}(P)$ 
for $P=\sum_{i=0}^{n}a_{i}(x)\partial^{i}\in W(x)$ $(a_{n}(x)\neq 0)$.
Let $\{(w'_{i},P'_{i})\mid i=1,\ldots,r'\}$ 
be the complete set of e-components  
of $r_{q'}\circ \phi_{\infty}(P)$ 
with a suitable $q'\in \mathbb{Z}_{>0}$.
Let $\lambda'_{1}<\lambda_{2}'<\ldots<\lambda'_{l'-1}$ be 
slopes of $N_{\infty}(P)$.
Then we have
\begin{equation}\label{NP infinity}
\begin{aligned}
    i_{l'}&=n=\mathrm{rank\,}P,\\
    j_{l'}&=n-\deg_{K[x]}a_{n}(x),\\
    i_{k+1}-i_{k}&=\sum_{\left\{i\,\Big|\,
        \frac{\delta(w'_{i})}{q'}=\lambda'_{k}\right\}}
    \mathrm{rank\,}P'_{i}\quad (k=1,\ldots,l'-1),\\
    j_{k+1}-j_{k}&=\lambda'_{k}\sum_{\left\{i
        \,\Big|\, \frac{\delta(w'_{i})}{q'}
    =\lambda'_{k}\right\}}\mathrm{rank\,}P'_{i}\quad (k=1,\ldots,l'-1).
\end{aligned}
\end{equation}
Moreover if $P\in W[x]$, 
then $\deg P$ can be seen from $N_{\infty}(P)$
as follows.
Choose $\alpha\,(\in\{1,\ldots,l'-1\})$ 
so that $\lambda'_{\alpha}>1$ and $\lambda'_{\alpha-1}\le 1$.
Then we have
\begin{equation}\label{degree from NP}
    \deg P=\deg_{K[x]}a_{n}(x)
    +\sum_{s=\alpha}^{l'-1}(\lambda'_{s}-1)
    \sum_{\left\{i\,\Big|\, \frac{\delta(w'_{i})}{q'}=\lambda'_{s}
\right\}}
    \mathrm{rank\,}P'_{i}.
\end{equation}
Also we can compute $v(\phi_{\infty}(P))$ as follows,
\begin{equation}\label{weight from NP}
    \begin{split}
        v(\phi_{\infty}(P))&=j_{1}=n-\deg_{K[x]}a_{n}(x)
        -\sum_{s=1}^{l'-1}\lambda_{s}
        \sum_{\left\{i\,\Big|\, \frac{\delta(w'_{i})}{q'}
        =\lambda'_{s}\right\}}
        \mathrm{rank\,}P'_{i}\\
        &=-\deg_{K[x]}a_{n}(x)-\sum_{s=1}^{l'-1}(\lambda'_{s}-1)
        \sum_{\left\{i\,\Big|\, \frac{\delta(w'_{i})}{q'}
        =\lambda'_{s}\right\}}
        \mathrm{rank\,}P'_{i}.
\end{split}
\end{equation}

\subsection{Spectral types} 
Let $P\in \widehat{W}(x)$ be regular singular type. 
Then it is known that there exists a base $\{u_{1},\ldots,u_{n}\}$
of  $\widehat{M}_{P}(x)$ as $K( (x))$-vector space 
and $A_{P}=(a_{i,j})_{\substack{1\le i\le n\\1\le j\le n}}\in M(n,K)$
such that
$
x\partial u_{i}=\sum_{j=1}^{n}a_{i,j}u_{j}\quad (i=1,\ldots,n).
$
Here we can normalize $A_{P}$ so that distinct eigenvalues 
of $A_{P}$ do not differ by integers 
(see Lemma 5.2.5 in \cite{S} for example). 
Let us call $A_{P}$ a \textit{characteristic matrix}.

\begin{df}[semisimple e-components]
    \normalfont
    Let us consider $P\in \widehat{W}(x)$ and an e-component $(w,Q)$ 
    of $P$ in $\widehat{W}(t)$, an extension of $\widehat{W}(x)$ with
    $t^{q}=x$.
    If a characteristic matrix of $Q$ is diagonalizable, then we say
    that $(w,Q)$ is a {\em semisimple e-component} of $P$.
\end{df}
\begin{df}\label{strongly semisimple}
    \normalfont
    Let us suppose that $P\in \widehat{W}(x)$ is unramified
    and the complete set of e-components of $P$ is 
    $\{(w_{1},P_{1}),\ldots,(w_{r},P_{r})\}$.
    Then we say $P$ has the {\em strongly semisimple decomposition}
    if the following are satisfied.
    \begin{enumerate}
        \item All e-components $(w_{i},P_{i})$, $i=1,\ldots,r$,
            are semisimple.
        \item There exist  $m^{[i]}_{j}\in \Z_{>0}$ 
              and $\lambda^{[i]}_{j}\in K$ $(j=1,\ldots,s_{[i]})$ 
              such that
              $\lambda^{[i]}_{j}-\lambda^{[i]}_{j'}\notin \Z$ $(j\neq j')$
              and the characteristic exponents of $P_{i}$ 
              for $i=1,\ldots,r$
            are
              \begin{align*}
                  &\lambda^{[i]}_{1},\lambda^{[i]}_{1}+1,\ldots,
                  \lambda^{[i]}_{1}+m^{[i]}_{1}-1,\\
                  &\quad \quad \cdots,\\
                  &\lambda^{[i]}_{s_{[i]}},\lambda^{[i]}_{s_{[i]}}+1,
                  \ldots,\lambda^{[i]}_{s_{[i]}}+m^{[i]}_{s_{[i]}}-1.
                \end{align*}
    \end{enumerate}
    We call the tuple of characteristic exponents and their multiplicities,
    \[
        \left\{\left(\lambda^{[i]}_{1},\ldots,
        \lambda^{[i]}_{s_{[i]}}\right)\,;
        \,\left(m^{[i]}_{1},\ldots,m^{[i]}_{s_{[i]}}\right)\right\}
    \]
    the {\em spectrum} of $P_{i}$ for each $i=1,\ldots,r$.
    Moreover the {\em spectrum} of $P$ is the set of spectra of $P_{i}$ for 
    $i=1,\ldots,r$.

    In particular if $P$ is  regular singular type, 
    then we simply say $P$ is {\em strongly semisimple}.
\end{df}
\begin{rem}
    \normalfont
    The condition 2 in the above definition  naturally appears in classical differential equations,
    for example the equations for generalized hypergeometric functions
    (see \cite{O} for instance).
    Let us see a trivial example. Take $P\in W(x)$ with $\mathrm{rank\,}P=n$ and $a\notin S_{P}$, then 
    \[
        ch_{a}(P)(s)=\textsl{constant}\cdot s(s-1)\cdots(s-n+1).
    \]
    Thus the characteristic exponents are $0,1,\ldots,n-1$.
\end{rem}
For $P\in \widehat{W}(x)$ there exist $p_{j}(s)\in K[s]$ and we can write 
\[
    P=\sum_{j=v}^{\infty}x^{j}p_{j}(\vartheta)\quad (p_{v}\neq 0)
\]
where $\vartheta=x\partial$. 
We see that $v=v(P)$ and $ch(P)(s)=p_{v}(s)$.
The conditions in Definition \ref{strongly semisimple} can be 
reformulated as conditions on $p_{j}(s)$.

\begin{prop}\label{spectypeoshima}
    We retain the notation in Definition \ref{strongly semisimple}.  
    Let us assume $P\in \widehat{W}(x)$ is unramified. 
    Let us choose $p^{[i]}_{l}(s)\in K[s]$ so that 
    \[
        P_{i}=\sum_{l=v_{[i]}}^{\infty}x^{l}p^{[i]}_{l}(\vartheta)
        \quad \left(p^{[i]}_{v[i]}\neq 0\right).
    \]
Then the conditions in Definition  \ref{strongly semisimple} are 
equivalent to the following conditions on $p^{[i]}_{l}(s)$.
 For each  
$P_{i}$ $(i=1,\ldots,r)$, 
there exist  $m^{[i]}_{j}\in \Z_{>0}$ 
and $\lambda^{[i]}_{j}\in K$ $(j=1,\ldots,s_{[i]})$ such that
$\lambda^{[i]}_{j}-\lambda^{[i]}_{j'}\notin \Z$ $(j\neq j')$
and
\begin{align*}
    &p^{[i]}_{v_{[i]}}(\lambda^{[i]}_{j})
    =p^{[i]}_{v_{[i]}}(\lambda^{[i]}_{j}+1)=\cdots
    =p^{[i]}_{v_{[i]}}(\lambda^{[i]}_{j}+m^{[i]}_{j}-1)=0, \\
    &p^{[i]}_{v_{[i]}+1}(\lambda^{[i]}_{j})
    =p^{[i]}_{v_{[i]}+1}(\lambda^{[i]}_{j}+1)=\cdots
    =p^{[i]}_{v_{[i]}+1}(\lambda^{[i]}_{j}+m^{[i]}_{j}-2)=0,\\
    &\cdots\\
    &p^{[i]}_{v_{[i]}+m^{[i]}_{j}-1}(\lambda^{[i]}_{j})=0
\end{align*}
for all $j=1,\ldots,s_{[i]}$.
\end{prop}
\begin{proof}
    This follows from Proposition 6.14 in \cite{O}.
\end{proof}
\begin{rem}\label{regular point}
    \normalfont
    Let us consider $P\in W(x)$ with $\mathrm{rank\,}P=n$ and take $a\in K\cup \{\infty\}$. 
    Then it can be seen that  $a\notin S_{P}$ if and only if
    $\phi_{a}(P)$ is regular singular type and 
    strongly semisimple with the spectrum
    $\{(0)\,;\,(n)\}$.
    We shall show this fact in the subsection \ref{primitive comp}.
\end{rem}
In this paper we mainly investigate special 
differential operators in $W(x)$ satisfying the following assumption.

\begin{ass}\label{spec}
    Let us consider $P\in W(x)$ with the set of singular points $S_{P}$. 
    Then for each $a\in S_{P}$ we assume that
    \begin{enumerate}
        \item $\phi_{a}(P)\in \widehat{W}(z)$ is unramified,
        \item $\phi_{a}(P)$ has the strongly semisimple decomposition.
    \end{enumerate}
       
\end{ass}
\begin{df}[spectral types]
    \normalfont
    Let us suppose that $P\in W(x)$ satisfies 
    \textsl{Assumption \ref{spec}} and 
    put 
    $\overline{S}_{P}=S_{P}\cup \{\infty\}=\{a_{0},a_{1},\ldots,a_{p}\}$
    where $a_{1},\ldots,a_{p}\in K$ and $a_{0}=\infty$.
    For each $a_{i}$ $(i=0,\ldots,p)$, let   
    $\left\{\left(w^{[i]}_1,P^{[i]}_{1}\right),\ldots,
    \left(w^{[i]}_{r_{i}},P^{[i]}_{r_{i}}\right)\right\}$
    be the complete set of e-components of $\phi_{a_{i}}(P)$.
    Let 
    \[
        \left\{\left(\lambda^{[i,j]}_{1},\ldots,
        \lambda^{[i,j]}_{s_{[i,j]}}\right)\,
        ;\,\left(m^{[i,j]}_{1},\ldots,
        m^{[i,j]}_{s_{[i,j]}}\right)\right\}\]
        be the spectrum of $P^{[i]}_{j}$ for each
        $i=0,\ldots,p$ and 
        $j=1,\ldots,r_{i}$.
        Then we say that $P$ has the {\em spectrum}
    \[
        \left\{\left(\lambda^{[i,j]}_{1},\ldots,
        \lambda^{[i,j]}_{s_{j}}\right)\,
        ;\,\left(m^{[i,j]}_{1},\ldots,
        m^{[i,j]}_{s_{[i,j]}}\right)\right\}
        _{\substack{0\le i\le p\\1\le j\le r_{i} }}.
    \]
    In particular, putting 
    $\mathbf{m}_{[i,j]}=\left(m^{[i,j]}_{1},\ldots,
    m^{[i,j]}_{s_{[i,j]}}\right)$, 
    we call 
    $(\mathbf{m}_{[i,j]})_{\substack{0\le i\le p\\1\le j\le r_{i}}}$
    the {\em spectral type} of $P$.
\end{df}

\section{Algebraic transformations and local data}
In this section, we recall some transformations on $W[x]$ and  $W(x)$, 
and investigate how the spectra are changed by these transformations.
\subsection{Addition and Fourier-Laplace transform}\label{add lap}
For $\mu\in K$ we  define an automorphism of $\widehat{W}(x)$ by 
\[
    \begin{array}{cccc}
        \mathrm{Add}^{\mu}\colon&\widehat{W}(x)&\longrightarrow&\widehat{W}(x)\\
        &x&\longmapsto&x\\
        &\partial&\longmapsto&\partial-\displaystyle\frac{\mu}{x}
    \end{array},
\]
and call this automorphism the \textit{addition}. 
Moreover define 
$\mathrm{Add}_{a}^{\mu}=
\phi_{a}^{-1}\circ \mathrm{Add}^{\mu}\circ \phi_{a}
\colon \widehat{W}(x)\rightarrow \widehat{W}(x)$ 
and call this the \textit{addition} at $x=a$ for $a\in K$.

\begin{lem}\label{shift}
    Let us consider an unramified $P\in\widehat{W}(x)$ 
    with the complete set of e-components   
    $\{(w_{1},P_{1})\ldots,(w_{r},P_{r})\}$.

   \begin{enumerate}
        \item For $g(x)\in K[ [x]]$, define an algebra automorphism 
            $j_{g}$ of 
            $\widehat{W}(x)$ by sending $x\mapsto x$ 
            and $\partial\mapsto \partial -g(x)$.
            Then $ch(j_{g}(P_{i}))=ch(P_{i})$ 
            and thus $j_{g}(P_{i})$ are  regular 
            singular type. 
            
            Moreover $j_{g}$ preserves characteristic matrices.
            Namely, we can choose the same characteristic matrices 
            of  $j_{g}(P_{i})$
            as them of $P_{i}$ for $i=1,\ldots,r$.

            In particular 
            $\left\{\left(w_{1},j_{g}(P_{1})\right),\ldots,
            \left(w_{r},j_{g}(P_{r})\right)\right\}$
            is the complete set of e-components of $j_{g}(P)$.
        \item
            For any $\mu\in K$,
            $\{(w_{1},\mathrm{Add}^{\mu}(P_{1})),\ldots,
            (w_{r},\mathrm{Add}^{\mu}(P_{r}))\}$ 
            is the complete set of e-components of 
            $\mathrm{Add}^{\mu}(P)$. 
     
            Moreover suppose that $P$ has the strongly 
            semisimple decomposition 
            with the spectrum 
            $\left\{\left(\lambda^{[j]}_{1},\ldots,
            \lambda^{[j]}_{s_{[j]}}\right)
            \,;\,
            \left(m^{[j]}_{1},\ldots,
            m^{[j]}_{s_{[j]}}\right)\right\}_{1\le j\le r}$.
            Then $\mathrm{Add}^{\mu}(P)$  has 
            the strongly semisimple decomposition with
            the spectrum 
            \[
                \left\{\left(\lambda^{[j]}_{1}+\mu,\ldots,
                \lambda^{[j]}_{s_{[j]}}+\mu\right)
            \,;\,
            \left(m^{[j]}_{1},\ldots,
            m^{[j]}_{s_{[j]}}\right)\right\}_{1\le j\le r}. 
            \]
    \end{enumerate}
\end{lem}
\begin{proof}
    First we note that $j_{g}(P_{i}^{w_{i}})=j_{g}(P_{i})^{w_{i}}$.
    Let us choose $p_{j}(s),p'_{j}(s)\in K[s]$ so that  
    $P_{i}=\sum_{j=v}^{\infty}x^{j}p_{j}(\vartheta)$ and 
    $j_{g}(P_{i})=\sum_{j=v'}^{\infty}x^{j}p'_{j}(\vartheta)$.
    Here $p_{v}(s)\neq 0$ and $p'_{v'}(s)\neq 0$. 
    Then $v(g(x))\ge 0$ implies that $v=v'$ and $p_{v}(s)=p'_{v'}(s)$.
    Thus $ch(P_{i})=ch(j_{g}(P_{i}))$.

    Next we examine characteristic matrices.
    Define a new $\partial$-action on $\widehat{M}_{P_{i}}$ by
    $\partial\circ m=(\partial+g(x))m$ for $m\in \widehat{M}_{P_{i}}$
    and denote this new $\widehat{W}(x)$-module 
    by $\widehat{M}_{P_{i}}^{g}$.
    Then $\widehat{M}_{P_{i}}^{g}\cong \widehat{M}_{j_{g}(P_{i})}(x)$.
    Let $A_{i}$ be a characteristic matrix of $P_{i}$ 
    with respect to a suitably chosen basis
    of $\widehat{M}_{P_{i}}$.  
    Then 
    \[
        \partial \circ m=(A_{i}+g(x)I)m
        \quad (m\in \widehat{M}^{g}_{P_{i}}).
    \]
    Here in the RHS we regard $m$ as the column vector with respect 
    to the basis
    and $I$ is the identity matrix. 
    Since we may choose $A_{i}$ so that 
    their eigenvalues do not differ by integers, 
    there exists 
    another basis of $\widehat{W}_{P_{i}}^{g}$ such that 
    \[
        \partial\circ m=A_{i}m \quad(m\in \widehat{M}^{g}_{P_{i}})
    \]
    (see Theorem 5.2.2 in \cite{S} for example).
    Here in the RHS we regard $m$ as the column vector with respect to 
    the new basis.
    This proves 1.

    Let us consider 2.  
    We see that $\mathrm{Add}^{\mu}(P_{i})\in \widehat{W}(x)$ is 
    regular singular type.
    Indeed if $P_{i}=\sum_{j=v}^{\infty}x^{j}p_{j}(\vartheta)$ 
    with polynomials $p_{j}(s)\in K[s]$ and $p_{v}(s)\neq 0$, 
    then $\deg_{K[t]}p_{v}(s)=\mathrm{rank\,} P$ 
    since $P$ is regular singular type.
    Then putting $p^{\mu}_{j}(s)=p_{j}(s-\mu)$, we see that
    \begin{equation}\label{shifting}
        \mathrm{Add}^{\mu}(P_{i})
        =\sum_{j=v}^{\infty}x^{j}p^{\mu}_{j}(\vartheta)
    \end{equation}
    and $\deg_{K[t]}p_{v}^{\mu}(s)=\deg_{K[s]}p_{v}(s)=\mathrm{rank\,}P
    =\mathrm{rank\,}\mathrm{Add}^{\mu}(P)$. 
    Thus $\mathrm{Add}^{\mu}(P_{i})$ is regular singular type.
    Moreover the last assertion follows from 
    the equation $(\ref{shifting})$ and Proposition \ref{spectypeoshima}.
\end{proof}

\begin{prop}\label{shiftII}
    Suppose that $P\in W(x)$ satisfies Assumption \ref{spec}. 
    Put $\overline{S}_{P}=\{a_{0},a_{1},\ldots,a_{p}\}$ 
    where $a_{0}=\infty$ and $a_{1},\ldots,a_{p}\in K$. 
    Let us denote the complete sets of e-components of $\phi_{a_{i}}(P)$ 
    by $\left\{\left(w^{[i]}_{1},P^{[i]}_{1}\right),\ldots,
    \left(w^{[i]}_{r_{i}},P^{[i]}_{r_{i}}\right)\right\}$
    and denote the spectra of $P^{[i]}_{j}$ by 
    $\left\{\left(\lambda^{[i,j]}_{1},\ldots,
    \lambda^{[i,j]}_{r_{i}}\right)\,;\,
    \left(m^{[i,j]}_{1},\ldots,
    m^{[i,j]}_{r_{i}}\right)\right\}$ 
    for $i=0,\ldots,p,\,j=1,\ldots,r_{i}$.
    Now take $\mu\in K$ and $a_{i_{0}}\in S_{P}$,
    and consider the spectrum of $\mathrm{Add}_{a_{i_{0}}}^{\mu}(P)$.

    Then we have $\overline{S}_{\mathrm{Add}_{a_{i_{0}}}^{\mu}(P)}
    \subset\overline{S}_{P}$. 
    For $a_{i}\in \overline{S}_{\mathrm{Add}_{a_{i_{0}}}^{\mu}(P)}$, 
    the complete sets of e-components of 
    $\phi_{a_{i}}\circ\mathrm{Add}_{a_{i_{0}}}^{\mu}(P)$
    are 
    $\left\{\left(w^{[i]}_{1},\tilde{P}^{[i]}_{1}\right),
    \ldots,\left(w^{[i]}_{r_{i}},
    \tilde{P}^{[i]}_{r_{i}}\right)\right\}$
    where $\tilde{P}^{[i]}_{j}$ have the spectra,
    \[
        \begin{cases}
            \left\{\left(\lambda^{[i,j]}_{1},\ldots,
            \lambda^{[i,j]}_{s_{[i,j]}}\right)\,
            ;\,\left(m^{[i,j]}_{1},\ldots,
            m^{[i,j]}_{s_{[i,j]}}\right)\right\}
            &\text{if }i\neq i_{0} \text{ and } i\neq 0,\\ 
            \left\{\left(\lambda^{[i_{0},j]}_{1}+\mu,\ldots,
            \lambda^{[i_{0},j]}_{s_{[i_{0},j]}}+\mu\right)\,
            ;\,\left(m^{[i_{0},j]}_{1},\ldots,
            m^{[i_{0},j]}_{s_{[i_{0},j]}}\right)\right\}
            &\text{if }i=i_{0},\\
            \left\{\left(\lambda^{[0,j]}_{1}-\mu,\ldots,
            \lambda^{[0,j]}_{s_{[0,j]}}-\mu\right)\,
            ;\,\left(m^{[0,j]}_{1},\ldots,
            m^{[0,j]}_{s_{[0,j]}}\right)\right\}
            &\text{if }i=0.
        \end{cases}
    \]
\end{prop}
\begin{proof}
     Note that if $a\neq a'$, 
     $\frac{\mu}{x-a}=\frac{\mu}{a'-a}+c_{1}(x-a')+c_{2}(x-a')^{2}+\cdots\in K[ [x-a']]$ where 
     $c_{i}=\frac{1}{i!}\frac{d^{i}}{dx^{i}}\frac{\mu}{x-a}|_{z=a'}$.
     Also note that putting  $z=x^{-1}$, 
     we see $\frac{d}{dx}-\frac{\mu}{x-a}=-z^{2}\frac{d}{dz}-\mu zg(z)$ 
     with some $g(z)\in K[ [z]]$ satisfying $g(0)=1$.
     As we note in Remark \ref{regular point} 
     if $a\notin  S_{P}$, then $\phi_{a}(P)$ 
     is regular singular type and strongly semisimple with the spectrum
     $\{(0)\,;\,(n)\}$ where $n=\mathrm{rank\,}P$. 
     Then by Lemma \ref{shift} 
     we can see that $\phi_{a}(\mathrm{Add}_{a_{i_{0}}}^{\mu})(P)$
     is regular singular type and strongly semisimple with the spectrum
     $\{(0)\,;\,(n)\}$.
     Thus $a\notin S_{\mathrm{Add}_{a_{i_{0}}}^{\mu}(P)}$ which shows 
     that $\overline{S}_{\mathrm{Add}_{a_{i_{0}}}^{\mu}(P)}
    \subset\overline{S}_{P}$.
    The other spectra of $\mathrm{Add}_{a_{i_{0}}}^{\mu}(P)$ can be 
    computed by Lemma \ref{shift}.
\end{proof}
      
\begin{df}[Fourier-Laplace transform]
    \normalfont
    The Fourier-Laplace transform is 
    the $K$-algebra automorphism of $W[x]$,
    \[
    \begin{array}{cccc}
        \mathcal{L}\colon&W[x]&\longrightarrow &W[x]\\
        &x&\longmapsto &-\partial\\
        &\partial&\longmapsto&x
    \end{array}.
    \]
\end{df}

We recall how spectra are changed by 
the Fourier-Laplace transform
following the results of J. Fang \cite{F} and C. Sabbah \cite{S1}.
\begin{prop}\label{finite to infinity}
    Let us consider $P\in W[x]$ and fix 
    $a\in S_{P}\backslash\{\infty\}\subset K$.
    Suppose that $\phi_{a}(P)\in \widehat{W}(z)$ is unramified
    and has the strongly semisimple decomposition.
    Moreover suppose that $\phi_{a}(P)$ has an e-component 
    $(w,Q)$  
    with $\delta(w)=n>0$ and the spectrum
    \[
    \{(\lambda_{1},\ldots,\lambda_{s})\,;
    \,(m_{1},\ldots,m_s)\}.
    \]
             
    Then there exist $\alpha_{1},\ldots,\alpha_{n+1}\in K$ 
    and distinct polynomials
    $g_{1}(x),\ldots, g_{n+1}(x)\in xK[x]$ of 
    $\deg g_{i}(x)=n$
    such that   
    $\phi_{\infty}(\mathcal{L}(P))\in\widehat{W}(z)$ 
    has the following e-components 
    $(v_{1},R_{1}),\ldots,(v_{n+1},R_{n+1})$ 
    in $\widehat{W}(t)$, 
    an extension of $\widehat{W}(z)$ with $t^{n+1}=z$. 
    \begin{enumerate}
    \item Polynomials $v_{i}$ are                
    \[
    v_{i}(t)=-at^{-n-1}+g_{i}(t^{-1}),\,
    i=1,\ldots,n+1.
    \]
\item We have $R_{i}$ are strongly semisimple with the spectra
    \[
        \{(\lambda_{1}+\alpha_{i},\ldots,\lambda_{s}+\alpha_{i})\,;\,
        (m_{1},\ldots,m_{s})\}
    \]
    for $i=1,\ldots,n+1$.
    \end{enumerate}
    Here $g_{i}(x)$ and $\alpha_{i}$ depend only on $w$.
\end{prop}
\begin{proof}
    Since the Fourier-Laplace transform sends the translation
    $x\mapsto x-a$  
    to $\partial\mapsto \partial+a$,
    it suffices to consider the case of $a=0$.

    Since we assume that $P$ has 
    the strongly semisimple decomposition, 
    Theorem 1.1 and 1.2 of J. Fang \cite{F}, Theorem 5.1 
     of C. Sabbah \cite{S1} and 
     Theorem in the section 1 of R. Garcia-Lopez  \cite{G} assure that
     any e-component of $\phi_{\infty}\circ\mathcal{L}(P)$
     whose slope $\lambda$ is $0<\lambda <1$ is semisimple.
    Thus we need to see the explicit characteristic exponents of 
    e-components of $\mathcal{L}(P)$.

 Set $P^{-w}=\sum_{i=0}^{N}a_{i}(x)(x^{n+1}\partial)^{i}$, then 
    \[
        P=\sum_{i=0}^{N}a_{i}(x)\left(x^{n+1}\partial-x^{n+1}
        \frac{d}{dx}w\right)^{i}.
    \]
    Putting $\tilde{w}(x)=x^{n+1}w=\sum_{i=0}^{n-1}w_{i}x^{i}$,
    we have
    \[
    \mathcal{L}(P)=
    \sum_{i=0}^{N}a_{i}(-\partial)( (-\partial)^{n+1}x
    -\tilde{w}(-\partial))^{i}\in W[x].
    \]
    Here we notice that $a_{i}(-\partial)$ are elements 
    in the ring of formal microlocal differential operators, 
    $\{\sum_{i\ge r}b_{i}(x)\partial^{-i}\mid b_{i}
    \in K[ [ x] ], r\in \Z\}$ (see \cite{G} for example).  
     
     We shall show that there exist polynomials 
     $g_{1},\ldots,g_{n+1}\in sK[x]$ 
     and $\alpha_{1},\ldots,\alpha_{n+1}\in K$ 
     such that we have
     \begin{equation}\label{localf}
         ch\left( (r_{n+1}\circ\phi_{\infty}\circ
         \mathcal{L}(P))^{-g_{i}(t^{-1})}\right)(s)
         =ch(P^{-w})(s-\alpha_{i}).
     \end{equation}
     This shows that
     there exist $R_{i}\in \widehat{W}(t)$, 
     the extension of $\widehat{W}( z)$ with $t^{n+1}=z$, 
     such that $(g_{i}(t^{-1}),R_{i})$ $(i=1,\ldots,n+1)$
     are e-components of $\phi_{\infty}\circ\mathcal{L}(P)$ 
     and 
     \begin{equation}\label{stsem}
         ch(R_{i})(s)=ch(P^{-w})(s-\alpha_{i})=ch(Q)(s-\alpha).
     \end{equation}
     
     Now let us show $(\ref{localf})$ and $(\ref{stsem})$. 
     Since 
    \[
        r_{n+1}\circ
        \phi_{\infty}\left((-\partial)^{n+1}x-\tilde{w}(-\partial) \right)
    = \left(\frac{1}{n+1}t^{n+2}\partial_{t}\right)^{n+1}t^{-(n+1)}
    -\tilde{w}\left(\frac{1}{n+1}t^{n+2}\partial_{t}\right),
    \]
    we put 
    $M(t,\partial_{t})=
    \left(\frac{1}{n+1}t^{n+2}\partial_{t}\right)^{n+1}t^{-(n+1)}
    -\tilde{w}\left(\frac{1}{n+1}t^{n+2}\partial_{t}\right)$.
    Then we have the following.
 \begin{lem}\label{uedesu}
    Let us retain the above notation.
    There exist $n+1$ polynomials 
    $h_{i}(x)=\sum_{j=1}^{n}h_{i,j}x^{j+1}$ $(i=1,\ldots,n+1)$ 
    and $\alpha_{1},\ldots,\alpha_{n+1}\in K$ 
    and we have 
    \begin{align}
        &M(x,\partial+h_{i}(x^{-1}))
        =\left(\frac{h_{i,n}}{n+1}\right)^{n}x^{n+1}
        \partial-\alpha_{i}x^{n}+S_{i}\\
        &-\frac{1}{n+1}x^{n+2}(\partial+h_{i}(x^{-1}))
        =\frac{h_{i,n}}{n+1}x+T_{i}.
    \end{align}
    Here $S_{i},T_{i}\in \widehat{W}(x)$ 
    with $v(S_{i})>n$ and $v(T_{i})>1$.
 \end{lem}
 \begin{proof}[proof of Lemma \ref{uedesu}]
     Let us put $h(x)=\sum_{i=1}^{n}h_{i}x^{i+1}\in K[x]$, then
     \[
         M(x,\partial+h(x^{-1}))
         =\left(\frac{1}{n+1}x^{n+2}\partial+\tilde{h}(x)
         \right)^{n+1}x^{-n-1}
         -\tilde{w}\left(\frac{1}{n+1}x^{n+2}\partial+\tilde{h}(x)\right).
         \]
     Here $\tilde{h}(x)=\frac{1}{n+1}x^{n+2}h(x^{-1})
     =\sum_{i=1}^{n}\tilde{h}_{i}x^{i}\in K[x]$.
     Since $v(x^{n+2}\partial)=n+1$ and $v(\tilde{h}(x))= 1$, 
     if  we put \[
     N_{i}(x,\partial)
     =\left(\frac{1}{n+1}x^{n+2}\partial+\tilde{h}(x)\right)^{i}
     -\tilde{h}(x)^{i}
     \]
     for $i=1,\ldots,n-1$, then $v_{0}(N_{i})=n+i$. Also putting  
     \[
     N_{n}(x,\partial)
     =\left(\frac{1}{n+1}x^{n+2}\partial+\tilde{h}(x)\right)^{n+1}
     x^{-n-1}-x^{-n-1}(\tilde{h}(x))^{n+1},
     \]
     we have $v_{0}(N_{n})=n$. 
     Then we have
     \[
         M(x,\partial+h(x^{-1}))=
         N_{n}-\sum_{i=1}^{n-1}w_{i}N_{i}-w_{0}
         +x^{-n-1}(\tilde{h}(x))^{n+1}
         -\sum_{i=1}^{n-1}w_{i}\, (\tilde{h}(x))^{i}.
     \]
     Let us put $(\tilde{h}(x))^{i}=\sum_{j=i}^{ni}H^{(i)}_{j}x^{j}$ 
     for $i=1,\ldots,n+1$. 
     Then we can see that $H^{(i)}_{i}=(\tilde{h}_{1})^{i}$ 
     and $H^{(i)}_{i+k}$ are polynomials of 
     $\tilde{h}_{1}, \ldots,\tilde{h}_{k}$ for $k=1,\ldots,n-1$. 
     
     Then let us choose $\tilde{h}_{i}\, (i=1,\ldots,n)$
     so that the following equations are satisfied,
         \begin{equation}
         \begin{split}
             &H^{(n+1)}_{n+1}-w_{0}=0\\
             &H^{(n+1)}_{n+2}-w_{1}H^{(1)}_{1}=0\\
         &\cdots\\
         &H^{(n+1)}_{n+1+j}-w_{1}H^{(1)}_{j}-\cdots -w_{j}H^{(j)}_{j}=0
         \ (j\le n-1)
     \end{split}.
     \end{equation}
     Then
     \[
         v\left(-w_{0}+x^{-n-1}(\tilde{h}(x))^{n+1}
         -\sum_{i=1}^{n-1}w_{i}\,(\tilde{h}(x))^{i}\right)\ge n,
     \]
     namely 
     \[
         -w_{0}+x^{-n-1}(\tilde{h}(x))^{n+1}
         -\sum_{i=1}^{n-1}w_{i}(\tilde{h}(x))^{i}
         =c_{0} x^{n}+c_{1}x^{n+1}+c_{2}x^{n+2}+\cdots.
     \]
     We note that the equation $(\tilde{h}_{n})^{n+1}-w_{0}=0$ 
     has $n+1$ solutions in $K$, and 
     if we fix a solution $\tilde{h}_{n}$, 
     remaining $\tilde{h}_{n-1},\ldots,\tilde{h}_{1}$ 
     are uniquely determined by the other equations.
 
     Thus we have
     \[
         M(x,\partial+h(x^{-1}))
         =(h_{n})^{n}x^{n+1}\partial-c_{0}x^{n}+M'(x,\partial)
     \]
     where $v(M'(x,\partial))\ge n+1$. 
     The other equation can be obtained similarly.
     \end{proof}
     Since
         $r_{n+1}\circ\phi_{\infty}\circ\mathcal{L}(P)
     =\sum_{i=0}^{N}a_{i}\left(\frac{1}{n+1}t^{n+2}\partial_{t}\right) 
     M(t,\partial_{t})^{i}
     $,                             
     let us choose $h_{k}$ as in Lemma \ref{uedesu} and $g_{k}$ 
     so that $\frac{d}{dt}(g_{k}(t^{-1}))=h_{k}(t^{-1})$
     for $k=1,\ldots,n+1$.
     Then putting $t_{k}=\frac{h_{k,n}}{n+1}t$,
     we have 
     $r_{n+1}\circ\phi_{\infty}\circ
     \mathcal{L}(P)^{-g_{k}(t^{-1})}=
     \sum_{i=0}^{N}a_{i}(t_{k}+T_{k})(t_{k}^{n+1}\partial_{t_{k}}
     -\alpha'_{k}t_{k}^{n}+S_{k})^{i}.$ 
     Here $\alpha'_{k}=\alpha_{k}\left(\frac{h_{k,n}}{n+1}\right)^{-n}$.
     Then recalling that 
     $P^{-w}=\sum_{i=0}^{N}a_{i}(x)(x^{n+1}\partial)^{i}$,
     we have
     $
         ch(\mathcal{L}(P)^{-g_{k}(t^{-1})})(t)
         =ch(P^{-w})(t-\alpha'_{k}).
     $
     Here $\alpha'_{k}=\alpha_{k}\left(\frac{h_{k,n}}{n+1}\right)^{-1}$.

 \end{proof}
 \begin{prop}\label{infinity to infinity}
     Let us consider $P\in W[x]$ and suppose that  
     $\phi_{\infty}(P)\in \widehat{W}(z)$ is unramified and 
     has the strongly semisimple decomposition. Moreover suppose that 
     $\phi_{\infty}(P)$ has an e-component
     $(w,Q)$
     with $\delta(w)=n\ge 2$ and  the spectrum 
     $$\{(\lambda_{1},\ldots,\lambda_{s})\,
     ;\,(m_{1},\ldots,m_s)\}.$$
     Set $w(z)=w_{0}z^{-n-1}+w_{1}z^{-n}+\ldots +w_{n-1}z^{-2}
     \ (w_{0}\neq 0)$  where $n\ge 2$. 
             
     Then there exist $\alpha_{1},\ldots,\alpha_{n-1}\in K$ 
     and distinct 
     $g_{1}(x),\ldots, g_{n-1}(x)\in xK[x]$ of 
     $\deg g_{i}(x)=n$
     such that $\phi_{\infty}(\mathcal{L}(P))$ has 
     following e-components $(v_{1},R_{1}),\ldots,(v_{n-1},R_{n-1})$ 
     in $\widehat{W}(t)$, the extension of $\widehat{W}(z)$ with 
     $t^{n-1}=z$.

     \begin{enumerate}
                 \item We have $R_{i}$ are of strongly 
                     semisimple with the spectra 
                     \[
                         \{(\lambda_1+\alpha_{i},\ldots,\lambda_{s}
                         +\alpha_{i})\,;\,(m_{1},\ldots,m_{s})\}.
                     \] 
                     for $i=1,\ldots,n-1$.

                 \item Polynomials $v_{i}$ are                  
                             \[
                                 v_{i}(t)=g_{i}(t^{-1}),\,i=1,\ldots,n-1.
                             \]
     \end{enumerate}
               Here $g_{i}(x)$ and $\alpha_{i}$ depend only on $w$.

 \end{prop}
 This proposition can be shown by the same argument 
 as in Proposition \ref{finite to infinity}. 
 Also we can show inversion formulas of these propositions.
 \subsection{Primitive component}\label{primitive comp}
 Elements in $W(x)$ can be seen as elements in $W[x]$ by 
 multiplying suitable elements in $K(x)$ from the left. 
 However $P\in W[x]$ and $f(x)P$ for some $f(x)\in K[x]$ have
 slightly different structures if we consider the Laplace transform images.
 For example, $\mathrm{rank\,}\mathcal{L}(P)\neq 
 \mathrm{rank\,}\mathcal{L}(f(x)P)$ and their local spectra are
 mutually different in general.
 Hence we shall give a way to choose a minimal, in a sense,
 element in $W[x]$
 from an element in $W(x)$.
 \begin{lem}\label{weight}
     Let us consider 
     \[
         P=\sum_{i=r}^{\infty}x^{i}p_{i}(\vartheta)\in 
         \widehat{W}[x]
         =K[ [x]][\partial]\ (p_{i}(s)\in K[s],\,p_{r}(s)\neq 0).
     \]
     Then $x^{-s}P$ is still in $\widehat{W}[x]$ 
     $(s\in \mathbb{Z}_{\ge 0})$
     if and only if $r-s\ge 0$ or 
     the following equations are satisfied for $m=s-r$,
     \begin{equation}\label{equ1}
     \begin{split}
         &p_{r}(0)=p_{r}(1)=\cdots=p_{r}(m-1)=0,\\
         &p_{r+1}(0)=p_{r+1}(1)=\cdots =p_{r+1}(m-2)=0,\\
         &\cdots\\
         &p_{r+m-1}(0)=0.
     \end{split}
     \end{equation}
 \end{lem}
 \begin{proof} 
     If equations $(\ref{equ1})$ are satisfied, we have
     \begin{align*}
         x^{r+i}p_{r+i}(\vartheta)&
         =x^{r+i}\vartheta(\vartheta-1)
         \cdots(\vartheta-m+i+1)\tilde{p}_{r+i}(\vartheta)\\
         &=x^{r+i}x^{m-i}\partial^{m-i}\tilde{p}_{r+i}(\vartheta)
         =x^{r+m}\partial^{m-i}\tilde{p}_{r+i}(\vartheta)
     \end{align*}
    for $i=0,1,\ldots,m-1$ 
    where $\tilde{p}_{r+i}\in K[x]$. 
    Thus $x^{-(r+m)}P\in \widehat{W}[x]$. 
 
    Conversely  
    suppose that $x^{-s}P\in \widehat{W}[x]$. 
    Then $x^{-s}P=\sum_{i=0}^{\infty}x^{i-m}p_{r+i}(\vartheta)$. 
    Since $v(x^{i-m}p_{r+i}(\vartheta))=i-m$, 
     they are linear combinations of 
     $x^{\alpha}\partial^{\alpha+m-i}$ $(\alpha\ge 0)$ for $i=0,\ldots,m$. 
     Recalling that 
     $$x^{\alpha}\partial^{\alpha+m-i}=\vartheta(\vartheta-1)
     \cdots(\vartheta-\alpha+1)\partial^{m-i}$$ 
     for $i=0,1,\ldots,m$, 
     we have
     \begin{align*}
         x^{i-m}p_{r+i}(\vartheta)&
         =\bar{p}_{r+i}(\vartheta)\partial^{m-i}\\
         &=\partial^{m-i}\bar{p}_{r+i}(\vartheta-m+i)\\
         &=x^{i-m}\vartheta(\vartheta-1)
         \cdots(\vartheta-m+i+1)\bar{p}_{r+i}(\vartheta-m+i)
     \end{align*}
     for $i=0,1,\ldots,m$. 
     Here $\bar{p}_{r+1}\in K[x]$. 
     Thus we have equations $(\ref{equ1})$.
 \end{proof}
 As a corollary of this lemma we show the fact in 
 Remark \ref{regular point}.
 \begin{prop}
 Let us consider $P\in W(x)$ with $\mathrm{rank\,}P=n$ and take $a\in K\cup \{\infty\}$. 
    Then $a\notin S_{P}$ if and only if
    $\phi_{a}(P)$ is regular singular type and 
    strongly semisimple with the spectrum
    $\{(0)\,;\,(n)\}$.
\end{prop}
\begin{proof}
    Multiplying an element in $K(x)$ from the left, 
    we may assume $P\in W[x]$.
    First let us suppose $a\notin S_{P}$.
    It suffices to consider the case $a=0$.
    Then multiplying $x^{m}$ from the left, 
    we may suppose
     $P=\sum_{i=0}^{n}a_{i}(x)\partial^{i}$ where 
     $a_{i}(x)\in K[x]$ and  $a_{n}(0)\neq 0$.
     Then $v(P)=-n$ and hence we can write 
     $P=\sum_{i=-n}^{N}x^{i}p_{i}(\vartheta)$
     with $p_{i}(s)\in K[s]$ and $\deg_{K[s]}p_{-n}(s)=n$.
     Then Lemma \ref{weight} shows that $\phi_{0}(P)$ is regular 
     singular
     type and strongly semisimple with the spectrum $\{(0)\,;\,(n)\}$.

     Let us suppose the converse. Lemma \ref{weight} shows that 
     by multiplying $x^{m}$, 
     we can set $P=\sum_{i=-n}^{N}x^{i}p_{i}(\vartheta)\in W[x]$ where 
     $p_{i}(s)$, $i=-n,-n+1,\ldots,-1$ satisfy 
     the equation $(\ref{equ1})$.
     Then $P=\sum_{i=0}^{n}a_{i}(x)\partial^{i}$ satisfies that 
     $a_{i}(x)\in K[x]$, $i=0,\ldots,n$ and $a_{n}(0)\neq 0$. 
     Thus $0\notin S_{P}$.
\end{proof}
      \begin{df}[primitive component]
          \normalfont
          We say that  
          $P=\sum_{i=0}^{n}a_{i}(x)\partial^{n}\in W[x]$ 
          is \textsl{primitive} if 
             \begin{enumerate}
                 \item 
                     $\mathrm{gcd}_{K[x]}\{a_{i}(x)\mid i=0,\ldots,n\}=1$,
                 \item $a_{n}(x)\neq 0$ is monic.
                 \end{enumerate}
                 For $P\in W(x)$, there exist $f(x)\in K(x)$ 
                 and the primitive element $\tilde{P}\in W[x]$, 
                 and then we can uniquely decompose $P$ as
                 \[
                 P=f(x)\tilde{P}.
                 \]
                 We denote this primitive element by $\mathrm{Prim}(P)$ 
                 and call this the {\em primitive component} of $P$.
         \end{df}
         Let us see some properties of primitive components.
         \begin{lem}\label{addprim}
             Let $P\in W[x]$ be a primitive element. Then take $a,\mu\in K$
             and decompose $\mathrm{Add}^{\mu}_{a}(P)=f(x)\mathrm{Prim\,}
             (\mathrm{Add}^{\mu}_{a}(P))$. 
             Then there exists $m\in \mathbb{Z}$ such that 
             $f(x)=x^{m}$.
         \end{lem}
         \begin{proof}
             We may assume $a=0$. 
             Set $\mathrm{Add}^{\mu}_{0}(P)
             =\sum_{i=0}^{n}a_{i}(x)\partial^{i}$.
             Then each $a_{i}(x)\in K(x)$ $(i=0,\ldots,n)$ has 
             pole only at $0$.
             Thus if we decompose 
             $f(x)$ as the product of linear factors, 
             $f(x)=x^{m_{0}}\prod_{j=1}^{m}(x-a_{i})^{m_{i}}$,
             then $m_{i}\le 0$ for all $i=1,\ldots,m$.
             Then Lemma \ref{shift} and Lemma \ref{weight} show that
             $m_{i}=0$ for all $i=1,\ldots,m$.
         \end{proof}
         The following proposition is owing to H. Tsai \cite{Ts} 
         which assures that 
         if $P\in W(x)$ generate the maximal ideal in $W(x)$ 
         and satisfies a good condition, then
         $\mathrm{Prim}(P)\in W[x]$ also generates 
         the maximal ideal in $W[x]$.
         \begin{prop}[Tsai \cite{Ts}]\label{Tsai}
             Let us consider $P\in W(x)$ and put
             $\{a_{1},\ldots,a_{p}\}=S_{P}\backslash\{\infty\}$.
             At each $x=a_{i}$ $(i=1,\ldots,p)$, let us write
             \[
             P=\sum_{j=r_{i}}^{N_{i}}(x-a_{i})^{j}p_{j}^{(i)}
             (\vartheta_{a_{i}})
             \]
             by integers $r_{i}$, $N_{i}$, 
             and polynomials 
             $p^{(i)}_{j}(s)\, (p^{(i)}_{r_{i}}(s)\neq 0)$. 
             Let us suppose that  there exist  
             $m_{i}\in \mathbb{Z}_{\ge 0}$ for $i=1,\ldots,p$ such that  
        \begin{equation}
     \begin{split}
         &p_{r_i}^{(i)}(0)=p_{r_{i}}^{(i)}(1)
         =\cdots=p_{r_{i}}^{(i)}(m_{i}-1)=0,\\
         &p_{r_{i}+1}^{(i)}(0)=p_{r_{i}+1}^{(i)}(1)
         =\cdots =p_{r_{i}+1}^{(i)}(m_{i}-2)=0,\\
         &\cdots\\
         &p^{(i)}_{r_{i}+m_{i}-1}(0)=0.
     \end{split}
     \end{equation}
             Here we put $p_{j}^{(i)}(s)=0$ if $j>N_{i}$. 
             Moreover we assume that characteristic polynomials 
             $ch_{a_{i}}(P)(s)=p_{r_{i}}^{(i)}(s)$ have no integer root
             other than $0,1,\ldots,m_{i}-1$.
     
            Then if $P$ is irreducible in $W(x)$, 
            i.e., $P$ generates the maximal left ideal of $W(x)$, 
            then the primitive component $\mathrm{Prim\,}(P)$ of $P$ 
            generates the maximal ideal of $W[x]$.
         \end{prop}
\begin{proof}
    This follows from Corollary 5.5 in \cite{Ts}.
\end{proof}
\begin{lem}\label{deg vs addition}                                                          
    Let $P\in W[x]$ be a primitive element.
    Suppose that $\phi_{0}(P)$ has an unramified 
    e-component $(0,Q)$ in $\widehat{W}(z)$ 
    with the spectrum
    \begin{equation}\label{eqqqq}
        \{(0,\lambda_{1},\ldots,\lambda_{l})\,;
        \,(m_{0},m_{1},\ldots,m_{l})\}.
    \end{equation}
    Then $P_{\lambda_{1}}
    =\mathrm{Prim\,}(\mathrm{Add}_{0}^{-\lambda_{1}}(P))$ 
    has the unramified e-component $(0,Q_{\lambda_{1}})$ 
    in $\widehat{W}(z)$ with the spectrum
    \[
    \{(-\lambda_{1},0,\lambda_{2}-\lambda_{1},
    \ldots,\lambda_{l}-\lambda_{1})\,
    ;\,(m_{0},m_{1},\ldots,m_{l})\}.
    \]
    Moreover
    \[
    \deg P_{\lambda_{1}}-\deg P=m_{0}-m_{1}.
    \]
\end{lem}
 \begin{proof}
     The first assertion follows from Proposition \ref{shift} 
     and the second one from Lemma \ref{weight} and Lemma \ref{addprim}.
 \end{proof}
\subsection{Fourier-Laplace transform of e-components with slope 
$\lambda\le 1$}
In the subsection \ref{add lap}, we see the Fourier-Laplace transform of 
e-components. 
However we exclude e-components of  the slope $\lambda=0$ 
in Proposition \ref{finite to infinity} 
and them of the slope $\lambda\le 1$ 
in Proposition \ref{infinity to infinity}.
Thus we shall see the remaining cases in this subsection.
\begin{prop}\label{kokokore}      
    Let $P\in W[x]$ be primitive.
    Fix $a\in S_{P}\backslash\{\infty\}$ and 
    suppose that $\phi_{a}(P)\in \widehat{W}(z_{a})$ has an unramified 
    e-component $(0,Q)$ 
    in $\widehat{W}(z_{a})$ with the spectrum
    \[
    \{(0,\lambda_{1},\ldots,\lambda_{l});(m_{0},m_{1},\ldots,m_{l})\}.
    \]
    Then $\phi_{\infty}(\mathcal{L}(P))\in \widehat{W}(z_{\infty})$ has the 
    unramified e-component 
    $(-az^{-1},Q')$ in $\widehat{W}(z_{\infty})$ with the spectrum
    \[
    \{(\lambda_{1}+1,\ldots,\lambda_{l}+1);(m_{1},\ldots,m_{l})\}.
    \]
\end{prop}
\begin{rem}
    In the above proposition we may allow the case $m_{0}=0$. 
    Thus we extend our notation for spectra as follows.
    Let us consider $(\lambda_{1},\ldots,\lambda_{r})\in K^{r}$ and 
    $(m_{1},\ldots,m_{r})\in (\mathbb{Z}_{\ge 0})^{r}$.
    Put $I=\left\{i\in \{1,\ldots,r\}\mid m_{i}\neq 0\right\}$
    and suppose $I\neq \emptyset$.
    Then we say that $P\in \widehat{W}(x)$ of regular singular type 
    and strongly semisimple has the spectrum
    $\{(\lambda_{1},\ldots,\lambda_{r})\,;\,(m_{1},\ldots,m_{r})\}$
    if $P$ has the spectrum 
    $\{(\lambda_{i})_{i\in I}\,;\,(m_{i})_{i\in I}\}$.
\end{rem}
 \begin{proof}
     We may consider only the case of $a=0$.  
     Set $P=\sum_{i=r}^{N}x^{i}p_{i}(\vartheta)$, $p_{r}(s)\neq 0$. 
     Lemma \ref{weight} tells us that $v(P)=r=-m_{0}$ and 
     \begin{align*}
         x^{-m_{0}+i}p_{-m_{0}+i}(\vartheta)&
         =x^{-m_{0}+i}\vartheta\cdots(\vartheta-m_{0}+i+1)\bar{p}_{-m_{0}+i}(\vartheta)\\
     &=x^{-m_{0}+i}x^{m_{0}-i}\partial^{m_{0}-i}
     \bar{p}_{-m_{0}+i}(\vartheta)\\
     &=\partial^{m_{0}-i}\bar{p}_{-m_{0}+i}(\vartheta),
     \end{align*}
     for $i=0,\ldots m_{0}-1$. 
     Here $\bar{p}_{j}(s)$ are polynomials. 
     If we denote valuations of $\widehat{W}(z_{0})$ and $\widehat{W}(z_{\infty})$
     by $v_{0}$ and $v_{\infty}$ respectively,
     we have $v_{0}(P)=v_{\infty}(\phi_{\infty}(\mathcal{L}(P)))$.
     Thus if we set 
     $\mathcal{L}(P)=\sum_{i=-m_{0}}^{N}x^{-i}\tilde{p}_{i}(\partial)$,
     then $\tilde{p}_{i}(s)=\bar{p}_{i}(-s-1)$ for $i=-1,-2,\ldots,-m_{0}$.
     Then the proposition follows.
 \end{proof}
Similarly we have the following.
 \begin{prop}\label{lap inv}
     Let us consider $P\in W[x]$. 
     Suppose that $\phi_{\infty}(P)\in \widehat{W}(z_{\infty})$ 
     has an e-component 
     $(az^{-1},Q)$ in $\widehat{W}(z_{\infty})$ with the spectrum
     \[
     \{(\lambda_{1},\ldots,\lambda_{l})\,;\,(m_{1},\ldots,m_{l})\}
     \]
     where $\lambda_{i}\notin\mathbb{Z}$ for $i=1,\ldots,l$.
     Then we have $v(\phi_{a}(\mathcal{L}(P)))\le 0$
     and   
     $\phi_{a}(\mathcal{L}(P))\in \widehat{W}(z_{a})$
     has the e-component
     $(0,Q')$ in $\widehat{W}(z_{a})$ with the spectrum
     \[
     \{(0,\lambda_{1}-1,\ldots,\lambda_{l}-1)\,;\,
     (m_{0},m_{1},\ldots,m_{l})\}.
     \]
     Here we put $-m_{0}=v(\phi_{a}(\mathcal{L}(P)))$.
 \end{prop}
 \begin{proof}
     Let us put 
     $\mathcal{L}(P)=\sum_{i=m_{a}}^{N}(x-a)^{i}p_{i}(\vartheta)$,
     $p_{i}(s)\in K[s]$ and $p_{m_{a}}\neq 0$,
     and suppose that $m_{a}=v(\phi_{a}(\mathcal{L}(P))> 0$.
     Then $P=\sum_{i=m_{a}}^{N}(\partial-a)^{i}p_{i}(-\vartheta-1)$
     and this implies that $P=(\partial-a)^{m_{a}}P'$ with $P'\in W[x]$.
     This shows that $ch(\phi_{\infty}(P)^{-az^{-1}_{\infty}})(s)=0$ has 
     integer roots and contradicts to the assumption $\lambda_{i}\notin
     \mathbb{Z}$ for all $i=1,\ldots,l$.
     The remaining follows from the same argument as in Proposition \ref{kokokore}.
 \end{proof}
 \subsection{Euler transform}
 In the previous subsections 
 we compute the changes of spectra
 of e-components by the addition and 
 the Fourier-Laplace transform.
 In this subsection we introduce Euler transform defined 
 by the composition of the Fourier-Laplace transform and the addition,
 and give an explicit computation of changes of spectra.
 
\begin{df}[Euler transform, cf. \cite{O}]
    \normalfont
    The {\em Euler transform} of $P\in W(x)$ 
    with the parameter $\lambda\in K$ is defined by 
     \[
     E(\lambda)P=
     \mathcal{L}\circ \mathrm{Prim}\circ 
     \mathrm{Add}^{\lambda}_{0}
     \circ \mathcal{L}^{-1}\circ\mathrm{Prim}(P)\in W[x].
     \]
\end{df}
\begin{rem}
    \normalfont
    This is an algebraic analogue of the classical description 
    of Euler transform:                           
    \[
    I^{\mu}_{c}g(z)
    =\frac{1}{\Gamma(\mu)}\int_{c}^{z}g(x)(z-x)^{\mu-1}\,dx
    =\int_{-i\infty}^{i\infty}y^{-\mu}
    \int_{c}^{\infty}g(x)e^{-xy}\,dx\, e^{zy}\,dy.
    \]
\end{rem}

\begin{thm}\label{Computation of Euler transform}
     Let us consider $P\in W(x)$ satisfying Assumption \ref{spec}
     and put $\overline{S}_{P}=\{a_{0},a_{1},\ldots,a_{p}\}$ 
     where $a_{0}=\infty$ and $a_{1},\ldots,a_{p}\in K$.
     Let 
     $\left\{\left(w_{1}^{[i]},P^{[i]}_{1}\right),\ldots,
     \left(w_{r_{i}},P^{[i]}_{r_{i}}\right)\right\}$
     be the complete set of e-components of 
     $\phi_{a_{i}}(P)\in \widehat{W}(z_{a_{i}})$
     for each $a_{i}\in \overline{S}_{P}$.
     Let us denote the spectra of $P^{[i]}_{j}$ by
     $$\left\{\left(\lambda^{[i,j]}_{1},\ldots,
     \lambda^{[i,j]}_{s_{[i,j]}}\right)\,;\,
     \left(m^{[i,j]}_{1},\ldots,
     m^{[i,j]}_{s_{[i,j]}}\right)\right\},
     \quad i=0,\ldots,p,\,j=1,\ldots,r_{i}.$$
     Moreover we assume the following.
     \begin{enumerate}
         \item Suppose that $w^{[i]}_{1}=0$ for $i=0,\ldots,p$ and 
             $\lambda^{[i,j]}_{1}=0$ for $i=1,\ldots,p$, 
             $j=1,\ldots,r_{i}$.
     \item Suppose that
         \begin{equation}\label{noninteger}
             \begin{aligned}
     \lambda^{[0,j]}_{k}&\notin \mathbb{Z}\quad 
     (k=1,\ldots,s_{[0,j]})
     \quad\text{if } \delta(w^{[0,j]})\le 1,\\
     \lambda_{k}^{[i,1]}+\lambda^{[0,1]}_{1}&\notin 
     \Z \quad \text{for all }i=1,\ldots,p,\,k=2,\ldots,s_{[i,1]}.
     \end{aligned}
 \end{equation}
     \end{enumerate}
 Let us put $\mu=1-\lambda^{[0,1]}_{1}$ for simplicity.
     Then we have the following.
     \begin{itemize}
         \item[{\em (i)}] We have 
             $\mathrm{rank\,}E(\mu)P=\mathrm{rank\,}P+d$
             where 
             \[
             d=\deg \mathrm{Prim}(P)-\sum_{j=1}^{s_{[0,1]}}
             m_{j}^{[0,1]}-m_{1}^{[0,1]}.
             \]
             We also have $m^{[i,1]}_{1}+d\ge 0$ for all $i=0,\ldots,p$.
         \item[{\em (ii)}] We have $S_{E(\mu)P}\subset S_{P}$ and 
             $E(\mu)(P)$ satisfies Assumption \ref{spec}.
         \item[{\em (iii)}] 
             At each $a_{i}$ $(i=0,\ldots,p)$,
             $\phi_{a_{i}}(E(\mu)P)$ has the 
             following complete set of e-components
             $\left\{\left(\tilde{w}^{[i]}_{1},\tilde{P}^{[i]}_{1}\right),
             \ldots,
             \left(\tilde{w}^{[i]}_{\tilde{r}_{i}},
             \tilde{P}^{[i]}_{\tilde{r}_{i}}\right)\right\}$.
             We have $\tilde{r}_{i}=r_{i} \ (i=0,\ldots,p)$
             and $\tilde{w}^{[i]}_{j}=w^{[i]}_{j}
             \ (i=0,\ldots,p,\,j=1,\ldots,r_{i}).$
             Here we change the order of 
             $\left(\tilde{w}^{[i]}_{j},\tilde{P}^{[i]}_{j}\right)$,
             $j=1,\ldots,r_{i}$, if necessary.
             Moreover each $\tilde{P}^{[i]}_{j}$ 
             $(i=1,\ldots,p,\,j=1,\ldots,r_{i})$ has 
             the spectrum
             \begin{align*}
             &\left\{\left(\lambda^{[i,j]}_{1}-\mu^{[i,j]},\ldots,
             \lambda^{[i,j]}_{s_{[i,j]}}-\mu^{[i,j]}\right)\,;\,
             \left(m^{[i,j]}_{1},\ldots,m^{[i,j]}_{s_{[i,j]}}\right)
         \right\}
             \quad \text{if }j\ge 2,\\
             &\left\{\left(0,\lambda^{[i,1]}_{2}-\mu,\ldots,
             \lambda^{[i,1]}_{s_{[i,1]}}-\mu\right)\,;\,
             \left(m^{[i,1]}_{1}+d,\ldots,
             m^{[i,1]}_{s_{[i,1]}}\right)\right\}
             \end{align*}
             where 
             $\mu^{[i,j]}=(\delta(w^{[i]}_{j})+1)\mu.$
             On the other hand 
             $\tilde{P}^{[0]}_{j}$ has the spectrum
             \begin{align*}
             &\left\{\left(\lambda^{[0,j]}_{1}-\mu^{[0,j]},\ldots,
             \lambda^{[0,j]}_{s_{[0,j]}}-\mu^{[0,j]}\right)\,;\,
             \left(m^{[0,j]}_{1},\ldots,
             m^{[0,j]}_{s_{[0,j]}}\right)\right\}
             \quad\text{if }j\ge 2,\\
             &\left\{\left(1+\mu,\lambda^{[0,1]}_{2}+\mu,\ldots,
             \lambda^{[0,1]}_{s_{[0,1]}}+\mu\right)\,;\,
             \left(m^{[0,1]}_{1}+d,\ldots,
             m^{[0,1]}_{s_{[0,1]}}\right)\right\}
             \end{align*}
             where 
             $\mu^{[0,j]}
             =\left(\delta(w^{[0]}_{j})-1\right)\mu.$
             \end{itemize}
 \end{thm}
 \begin{proof}
    First we show (i). Let us denote the valuations 
     of $\widehat{W}(z_{a_{i}})$ by $v_{a_{i}}$ for 
     $i=0,\ldots,p$.
     Then from the equations $(\ref{NP finite})$ 
     and Lemma \ref{weight}
     we have
     \[
         v_{a_{i}}(a_{N}(x))-N=
         \sum_{j=2}^{r_{i}}\delta(w^{[i]}_{j})
         \sum_{k=1}^{s_{[i,j]}}m^{[i,j]}_{k}-m^{[i,1]}_{1}.
     \]
     Since $N=\sum_{j=1}^{r_{i}}\sum_{k=1}^{s_{[i,j]}}m^{[i,j]}_{k}$, 
     then
     \[
         v_{a_{i}}(a_{N}(x))=
         \sum_{j=2}^{r_{i}}(\delta(w^{[i]}_{j})+1)
         \sum_{k=1}^{s_{[i,j]}}m^{[i,j]}_{k}
         +\sum_{k'=1}^{s_{[i,1]}}m^{[i,1]}_{k'}
         -m^{[i,1]}_{1}.
     \]
     Let us note that $a_{N}(x)$ has zeros only at $a_{1},\ldots,a_{p}$. 
     Thus $\deg_{K[x]}{a_{N}(x)}=\sum_{i=1}^{p}v_{a_{i}}(a_{N}(x)).$
     On the other hand, from the equation $(\ref{degree from NP})$,
     \[
         \mathrm{deg\,}\mathrm{Prim}(P)=\mathrm{deg\,}_{K[x]}a_{N}(x)+
         \sum_{j=2}^{r_{0}}(\delta(w^{[0]}_{j})-1)
         \sum_{k=1}^{s_{[0,j]}}m^{[0,j]}_{k}.
     \]
     Combining these formulas, we have
     \begin{align*}
         \mathrm{deg\,}&\mathrm{Prim}(P)=\\
         &
         \sum_{i=1}^{p}\sum_{j=2}^{r_{i}}\sum_{k=1}^{s_{[i,j]}}
         (\delta(w^{[i]}_{j})+1)m^{[i,j]}_{k}
         +\sum_{j'=1}^{r_{0}}\sum_{k'=1}^{s_{[0,j']}}
         (\delta(w^{[0]}_{j'})-1)m^{[0,j']}_{k'}\\
         &+\sum_{i'=1}^{p}\sum_{k'=1}^{s_{[i',1]}}m^{[i',1]}_{k'}
         -\sum_{i'=1}^{p}m^{[i',1]}_{1}.
     \end{align*}

     We can see that $\mathcal{L}^{-1}(\mathrm{Prim\,}(P))\in W[x]$ 
     has the e-component $(0,Q)$ in $\widehat{W}(x)$ with the spectrum 
     \[
         \left\{\left(0,\lambda^{[0,1]}_{1}-1,\ldots,
         \lambda^{[0,1]}_{s_{[0,1]}}-1\right)
         \,;\,\left(N_{0},m^{[0,1]}_{1},\ldots,
         m^{[0,1]}_{s_{[0,1]}}\right)\right\}
     \]
     where $N_{0}=\deg \mathrm{Prim\,}(P)
     -\sum_{j=1}^{s_{[0,1]}}m_{j}^{[0,1]}$ 
     by Proposition \ref{lap inv},
     the equation $v(\phi_{0}(\mathcal{L}(P)))=v(\phi_{\infty}(P))$
     and the equations $(\ref{degree from NP}),
     \,(\ref{weight from NP})$.
     By Proposition \ref{lap inv} and the assumption $(\ref{noninteger})$,
     we have $N_{0}\ge 0$.
     We note that $\mathcal{L}^{-1}\circ\mathrm{Prim}(P)$ 
     is the primitive element. 
     Indeed if there exist $f(x)(\neq 0)\in K[x]$ 
     and $R\in W[x]$ such that 
     \[
     \mathcal{L}^{-1}(\mathrm{Prim\,}(P))=f(x)R,
     \]
     then $P$ can be divided by
     \[
     f(-\partial)=C(\partial-\alpha_{1})\cdots (\partial-\alpha_{k})
     \]
     for some constants  $C,\alpha_{1},\ldots,\alpha_{k}\in K$. 
     However this means that $\phi_{a_{0}}(P)$ has e-components
     $(\alpha_{1}z_{a_{0}}^{-1},Q_{1}),\ldots,
     (\alpha_{k}z_{a_{0}}^{-1},Q_{k})$
     in $\widehat{W}(z_{a_{0}})$ 
     whose characteristic polynomials have integer roots. 
     This contradicts to the assumption $(\ref{noninteger})$.   
     
     Hence 
     $Q_{\mu}=
     \mathrm{Prim\,}\circ\mathrm{Add}^{\mu}\circ\mathcal{L}^{-1}
     \circ\mathrm{Prim\,}(P)$ 
     has the e-component $(0,\mathrm{Add}^{\mu}Q)$ in $\widehat{W}(x)$ 
     with the spectrum
     \[
         \left\{\left(\mu,0,\ldots,
             \lambda^{[0,1]}_{s_{[0,1]}}+\mu-1\right)\,
         ;\,
         \left(N_{0},m^{[0,1]}_{1},\ldots,
         m^{[0,1]}_{s_{[0,1]}}\right)\right\}
     \]
     and
     \begin{align*}
         \deg Q_{\mu}&=\deg \mathcal{L}^{-1}\circ\mathrm{Prim}(P)
         +N_{0}-m_{1}^{[0,1]}\\
         &=\mathrm{rank\,} P+N_{0}-m_{1}^{[0,1]}
     \end{align*}
     by Lemma \ref{deg vs addition}. Thus 
     \[
         \mathrm{rank\,}E(\mu)P=\deg Q_{\mu}
         =\mathrm{rank\,}P+N_{0}-m_{1}^{[0,1]}=\mathrm{rank\,}P+d.
     \]
     Here $d=N_{0}-m^{[0,1]}_{1}
     =\deg\mathrm{Prim}(P)
     -\sum_{j=1}^{s_{[0,1]}}m^{[0,1]}_{j}-m^{[0,1]}_{1}$.

     Next we see that $m^{[i,1]}_{1}+d\ge 0$ for all $i=0,\ldots,p$. 
     The case $i=0$ follows from $N_{0}\ge 0$. Thus we see the cases
     $i=1,\ldots,p$.
     By Lemma \ref{addprim}, $Q_{\mu}=x^{N_{\nu}}
     \mathrm{Add}^{\mu}\circ\mathcal{L}^{-1}
     \circ\mathrm{Prim\,}(P)$ where
     $N_{\mu}=\deg Q_{\mu}-\deg\mathcal{L}\circ\mathrm{Prim}(P)
     =N_{0}-m^{[0,1]}_{1}=d$.
     Hence we have $v(\phi_{a}(Q_{\mu}))-v(\phi_{a}(\mathcal{L}^{-1}
     \circ\mathrm{Prim}(P)))=N_{\mu}=d$ 
     for any $a\in K\backslash\{0\}$.
     Moreover we have $v(\phi_{a_{i}}(Q_{\mu}))\ge 0$ for $i=1,\ldots,p$
     by the assumption $(\ref{noninteger})$,
     Proposition \ref{kokokore}
     and \ref{lap inv}.
     Thus $0\le v(\phi_{a_{i}}(Q_{\mu}))=v(\phi_{a}(\mathcal{L}^{-1}
     \circ\mathrm{Prim}(P)))+d=m^{[i,1]}_{1}+d$ for $i=1,\ldots,p$
     as required.   
     Hence we obtain $\textrm{(i)}$.  
     
     Let us see the Euler transform preserves 
     the locations of singular points.
     Let us put $\mathrm{Prim}(P)=\sum_{i=0}^{N}a_{i}(x)\partial^{i}$ 
     ,$a_{N}(x)\neq 0$.
     Then $E(\mu)P
     =\partial^{m}\sum_{i=0}^{N}a_{i}(x+\mu\partial^{-1})\partial^{i}
     =\sum_{i=0}^{N'}a'_{N'}(x)\partial^{i}$
     for some $m\in \mathbb{Z}$ and $N'\in \mathbb{Z}_{\ge 0}$.
     Thus $a_{N}(x)=a'_{N'}(x)$ and 
     \[S_{P}\backslash\{\infty\}
         =\{\text{roots of }a_{N}(x)=0\}
         =\{\text{roots of }a'_{N'}(x)=0\}\supset
     S_{E(\mu)P}\backslash\{\infty\}.
     \]

     Also the remaining of $\text{(ii)}$ and $\text{(iii)}$ directly follow from Proposition \ref{shiftII},
     \ref{finite to infinity}, \ref{infinity to infinity}, \ref{kokokore}
     and \ref{lap inv}.
 \end{proof}
 \section{Euler transform and the action of Weyl groups
 of Kac-Moody root systems}
 \label{Euler transform and Weyl group}
In the previous section, we compute the changes of spectra
given by additions, the Fourier-Laplace transform and 
the Euler transform.
We shall see that the changes of the spectra
induce automorphisms of a $\mathbb{Z}$-lattice and these
automorphisms generate a transformation group of this lattice.
Moreover we shall see the lattice with the transformation group
is isomorphic to a quotient lattice of a Kac-Moody root lattice 
with the Weyl group action.

\subsection{Lattice transformations induced from the Euler transform}
Take $P\in W[x]$ satisfying Assumption \ref{spec}.
Set $\{a_{0},a_{1},\ldots,a_{p}\}=\overline{S}_{P}$ 
with $a_{0}=\infty$, $a_{1},\ldots,a_{p}\in K$.
For each $a_{i}\in \overline{S}_{P}$, 
let $$\left\{\left(w^{[i]}_{1},P^{[i]}_{1}\right),\ldots,
\left(w^{[i]}_{r_{i}},P^{[i]}_{r_{i}}\right)\right\}$$ be the 
complete set of e-components of $P$ with the spectra
$$\left\{\left(\lambda^{[i,j]}_{1},\ldots,
\lambda^{[i,j]}_{s_{[i,j]}}\right);
\left(m^{[i,j]}_{1},\ldots,
m^{[i,j]}_{s_{[i,j]}}\right)\right\}$$
for $j=1,\ldots,r_{i}$.

Define
\[
    \mathcal{J}(P)=\{1,\ldots,r_{0}\}\times\cdots\times\{1,\ldots,r_{p}\}.
\]
We fix the above $P\in W[x]$ all through
Section \ref{Euler transform and Weyl group}.
\begin{df}[twisted Euler transforms]
    \normalfont
    Consider the above $P\in W[x]$.
    Then for $\mathbf{j}=(j_{0},\ldots,j_{p})\in \mathcal{J}(P)$, 
    we define the {\em twisted Euler transform} $E(\mathbf{j})$ by
    \begin{align*}
        E(\mathbf{j})P&=
        \prod_{i=0}^{p}
        \mathrm{Ad\,}_{a_{i}}(e^{w^{[i]}_{j_{i}}})
        \prod_{i'=1}^{p}
        \mathrm{Add}_{a_{i'}}^{\lambda^{[i',j_{i'}]}_{1}}
        \ \circ E(1-\lambda_{\mathbf{j}})\\ 
        &\ \circ \prod_{i=0}^{p}
        \mathrm{Ad\,}_{a_{i}}(e^{-w^{[i]}_{j_{i}}})
        \prod_{i'=1}^{p}
        \mathrm{Add}_{a_{i'}}^{-\lambda^{[i',j_{i'}]}_{1}}
        (P).
    \end{align*}
    Here $\lambda_{\mathbf{j}}=\sum_{i=0}^{p}\lambda^{[i,j_{i}]}_{1}$.
\end{df}
Then by Theorem \ref{Computation of Euler transform} we can compute 
the spectra of $E(\mathbf{j})P$ for $\mathbf{j}\in \mathcal{J}(P)$.
\begin{prop}\label{Twisted Euler transform}
    Let us consider the above $P\in W[x]$.
%
    Suppose that there exists $\mathbf{j}=(j_{0},\ldots,j_{p})
    \in \mathcal{J}(P)$ such that
    \begin{equation}\label{generic}
   \begin{aligned}
        &\lambda^{[0,j]}_{k}+
        \sum_{i=1}^{p}\lambda^{[i,j_{i}]}_{1}\notin \mathbb{Z}
        ,\ (k=1,\ldots,s_{[0,j]}),
        \text{ if }\delta(w^{[0]}_{j}-w^{[0]}_{j_{0}})\le 1,\\
        &\lambda^{[i,j_{i}]}_{k}
        -\lambda^{[i,j_{i}]}_{1}
        +\lambda_{\mathbf{j}}\notin\mathbb{Z},
        \ (i=1,\ldots,p,\,k=2,\ldots,s_{[i,j_{i}]}).
    \end{aligned}
\end{equation}
    Let us fix this $\mathbf{j}$ and   
    put $P(\mathbf{j})=E(\mathbf{j})P$.
    Let 
    \[
    \left\{\left(w^{[i]}_{1},P(\mathbf{j})^{[i]}_{1}\right),\ldots,
    \left(w^{[i]}_{r_{i}},P(\mathbf{j})^{[i]}_{r_{i}}\right)\right\}
    \]
    be the complete sets of e-components of 
    $P(\mathbf{j})$ for $i=0,\ldots,p$
    and we denote the  spectra of $P(\mathbf{j})^{[i]}_{j}$ by
    \[
    \left\{\left(\lambda(\mathbf{j})^{[i,j]}_{1},\ldots,
    \lambda(\mathbf{j})^{[i,j]}_{s_{[i,j]}}\right);
    \left(m(\mathbf{j})^{[i,j]}_{1},\ldots,
    m(\mathbf{j})^{[i,j]}_{s_{[i,j]}}\right)\right\}
    \]
    for $i=0,\ldots,p$ and $j=1,\ldots,r_{i}$.
    Then we have
            \begin{align*}
                m(\mathbf{j})^{[i,j]}_{1}&=m^{[i,j]}_{1}+d(\mathbf{j})
                &\text{if }j=j_{i},\\
                m(\mathbf{j})^{[i,j]}_{k}&=m^{[i,j]}_{k}
                &\text{otherwise},
            \end{align*}
            where
            \begin{align*}
                d(\mathbf{j})&
                =\sum_{i=1}^{p}\sum_{j=1}^{r_{i}}
                (\delta(w^{[i]}_{j}-w^{[i]}_{j_{i}})+1)
              \sum_{k=1}^{s_{[i,j]}}m^{[i,j]}_{k}\\
               &\quad 
               +\sum_{j=1}^{r_{0}}(\delta(w^{[0]}_{j}
               -w^{[0]}_{j_{0}})-1)
               \sum_{k=1}^{s_{[0,j]}}m^{[0,j]}_{k}
               -\sum_{i=0}^{p}m^{[i,j_{i}]}_{1}.
            \end{align*}
            Also we have that for 
            $i=1,\ldots,p,\,j=1,\ldots,r_{i},\,k=1,\ldots,s_{[i,j]}$, 
            \begin{align*}
                &\lambda(\mathbf{j})^{[i,j]}_{k}=\\
                &\quad\quad\begin{cases}\lambda^{[i,j]}_{k}
                    &\text{if }j=j_{i}\text{ and }k=1,\\
                    \lambda^{[i,j]}_{k}-
                    (\delta(w^{[i]}_{j}-w^{[i]}_{j_{i}})+1)
                    (1-\lambda(\mathbf{j}))&\text{otherwise},
            \end{cases}  
            \end{align*}
            and for $i=0,\,j=1,\ldots,r_{0},\,k=1,\ldots,s_{[0,j]}$,
            \begin{align*}
                &\lambda(\mathbf{j})^{[0,j]}_{k}=\\
            &\quad\quad\begin{cases}
                \lambda^{[0,j_{0}]}_{1}+2(1-\lambda(\mathbf{j}))
                &\text{if }j=j_{0}\text{ and }k=1,\\
                \lambda^{[0,j]}_{k}
                -(\delta(w^{[0]}_{j}-w^{[0]}_{j_{0}})-1)
                (1-\lambda(\mathbf{j}))&\text{otherwise}.
            \end{cases}
        \end{align*}
\end{prop}
\begin{proof}
    By the equation $(\ref{degree from NP})$, we have 
\begin{align*}
&\deg\bigg(
\mathrm{Prim\,} \prod_{i=1}^{p}
\mathrm{Add}_{a_{i}}^{-\lambda^{[i,j_{i}]}_{1}}
\prod_{i=0}^{p}
\Ad_{a_{i}}(e^{-w^{[i]}_{j_{i}}})
(P)\bigg)
-\sum_{k=1}^{s_{[0,j_{0}]}}m^{[0,j_{0}]}_{k}-m^{[0,j_{0}]}_{1}\\
&\quad=\sum_{i=1}^{p}\sum_{j=1}^{r_{i}}
(\delta(w^{[i]}_{j}-w^{[i]}_{j_{i}})+1)
\sum_{k=1}^{s_{[i,j]}}m^{[i,j]}_{k}\\
&\quad\quad+\sum_{j=1}^{r_{0}}
(\delta(w^{[0]}_{j}-w^{[0]}_{j_{0}})-1)
\sum_{k=1}^{s_{[0,j]}}m^{[0,j]}_{k}-\sum_{i=0}^{p}m^{[i,j_{i}]}_{1}.
\end{align*}
Thus the proposition follows from Theorem $\ref{Computation of Euler transform}$.
\end{proof}

The following proposition assures that the irreducibility is preserved by 
the Euler transform under a good condition.
\begin{prop}\label{irreducibility is preserved}
    Let us consider $P\in W[x]$ with $\mathrm{rank\,}P>1$ 
    and suppose that 
    there exists $\mathbf{j}\in \mathcal{J}(P)$ satisfying
    the condition $(\ref{generic})$ in Proposition \ref{Twisted 
    Euler transform}.
      
      Then if $P$ is irreducible in $W(x)$, 
      i.e., the ideal generated by 
      $P$ in $W(x)$ is a maximal ideal, then
      $E(\mathbf{j})P$ is irreducible in $W(x)$.
\end{prop}
\begin{proof}
    Put
    $
        P_{\mathbf{j}}=
    \mathrm{Prim}
    \prod_{i=1}^{p}
    \mathrm{Add}_{a_i}^{-\lambda_{1}^{[i,j_{i}]}}
    \prod_{i'=0}^{p}
    \Ad_{a_{i'}}(e^{-w_{j_{i'}}^{[i']}})
    (P).
    $
    Then Proposition \ref{Tsai} implies that $P_{\mathbf{j}}$
    generates the maximal ideal of $W[x]$. 
    Thus $\mathcal{L}^{-1}(P_{\mathbf{j}})$ also generates 
    the maximal ideal of $W[x]$. 
    Thus if $\mathcal{L}^{-1}(P_{\mathbf{j}})\notin K[x],$ 
    then $\mathcal{L}^{-1}(P_{\mathbf{j}})$ 
    generates the maximal ideal in $W(x)$.   
    Suppose conversely that
    $\mathcal{L}^{-1}(P_{\mathbf{j}})=f(x)\in K[x]$.
    Then $P_{\mathbf{j}}=f(-\partial)$. 
    Since $P_{\mathbf{j}}$ is irreducible, 
    $\deg_{K[x]}f(x)=1$. 
    This contradicts to the assumption $\mathrm{rank\,}P>1$. 
    Hence $\mathcal{L}^{-1}(P_{\mathbf{j}})$ is irreducible. 
    Since operators $\mathrm{Add}^{\mu}_{a}$ preserve
    the irreducibility,
    $\mathrm{Prim\,}\mathrm{Add}_{0}^{1-\lambda_{\mathbf{j}}}
    \mathcal{L}^{-1}(P_{\mathbf{j}})$ generates the maximal ideal 
    in $W[x]$ by Proposition \ref{Tsai}.
    Thus if $E(1-\lambda_{\mathbf{j}})(P_{\mathbf{j}})\notin K[x]$,
    then $E(1-\lambda_{\mathbf{j}})(P_{\mathbf{j}})$ is irreducible
    which implies that  $E(\mathbf{j})(P)$ is irreducible.
    Suppose that $E(1-\lambda_{\mathbf{j}})(P_{\mathbf{j}})=g(x)\in K[x]$.
    Then $\mathrm{Prim\,}\mathrm{Add}_{0}^{1-\lambda_{\mathbf{j}}}
    \mathcal{L}^{-1}(P_{\mathbf{j}})=g(\partial).$
    Hence $g(x)=ax+b$ for some $a,b\in K$ from the irreducibility. 
    Then there exists $f(x)\in K[x]$ such that 
     $
         \mathrm{Add}_{0}^{1-\lambda_{\mathbf{j}}}
         \mathcal{L}^{-1}(P_{\mathbf{j}})
         =f(x)(ax\partial+bx).
    $
    Thus $P_{\mathbf{j}}
    =f(-\partial)(-ax\partial-b\partial-1+\lambda_{\mathbf{j}}).$
    This is a contradiction since $\mathrm{rank\,}P>1$
    and $P$ is irreducible.
\end{proof}

As we see in Proposition \ref{Twisted Euler transform}, 
the twisted Euler transforms $E(\mathbf{j})$ change
the spectral type $(m^{[i,j]}_{k})
_{\substack{0\le i\le p,\,1\le j\le r_{i},\\
1\le k\le s_{[i,j]}}}$,
the tuple of integers. 
This can be extended to $\mathbb{Z}$-lattice transformations
as follows.
Define the $\Z$-lattice,
\begin{align*} 
    &L(P)=\\
    &\left\{\left(a^{[i,j]}_{1},\ldots,
    a^{[i,j]}_{s_{[i,j]}}\right)
    _{\substack{0\le i\le p\\1\le j\le r_{i}} }\,\bigg|\,
    a^{[i,j]}_{k}\in\Z,\ 
    \sum_{j=1}^{r_{0}}\sum_{k=1}^{s_{[0,j]}}a^{[0,j]}_{k}
    =\cdots
    =\sum_{j=1}^{r_{p}}\sum_{k=1}^{s_{[p,j]}}a^{[p,j]}_{k}
\right\}. 
\end{align*}
The rank of 
$\mathbf{a}=(a^{[i,j]}_{1},\ldots,a^{[i,j]}_{s_{[i,j]}})
_{\substack{0\le i\le p\\1\le j\le r_{i}}}\in L(P)$ is
defined by
$\mathrm{rank\,}(\mathbf{a})
=\sum_{j=1}^{r_{i}}\sum_{k=1}^{s_{[i,j]}}a^{[i,j]}_{k}.
$
We denote the set of positive elements in $L(P)$ by $L(P)^{+}$, i.e.,
all components of the elements in $L(P)^{+}$ are in $\mathbb{Z}_{\ge 0}$.
For each $\mathbf{j}=(j_{0},\ldots,j_{p})\in \mathcal{J}(P)$, 
let us define  a $\Z$-automorphism of $L(P)$ by
\[
\begin{array}{cccc}
    \sigma(\mathbf{j})\colon&L(P)&\longrightarrow&L(P)\\
    &\mathbf{a}=\left(a^{[i,j]}_{1},\ldots,
    a^{[i,j]}_{s_{[i,j]}}\right)
    _{\substack{0\le i\le p\\1\le j\le r_{i}}}
    &\longmapsto
    &\tilde{\mathbf{a}}=\left(\tilde{a}^{[i,j]}_{1},\ldots,
    \tilde{a}^{[i,j]}_{s_{[i,j]}}\right)
    _{\substack{0\le i\le p\\1\le j\le r_{i}}}
\end{array}
\]
where \begin{align*}
\tilde{a}^{[i,j]}_{1}&=a^{[i,j]}_{1}+d(\mathbf{a};\mathbf{j})
&\text{if }j=j_{i},\\
\tilde{a}^{[i,j]}_{k}&=a^{[i,j]}_{k}&\text{otherwise},
\end{align*}
and
\begin{align*}
d(\mathbf{a};\mathbf{j})&=
\sum_{i=1}^{p}\sum_{j=1}^{r_{i}}
(\delta(w^{[i]}_{j}-w^{[i]}_{j_{i}})+1)
\sum_{k=1}^{s_{[i,j]}}a^{[i,j]}_{k}\\
&\quad +\sum_{j=1}^{r_{0}}
(\delta(w^{[0]}_{j}-w^{[0]}_{j_{0}})-1)
\sum_{k=1}^{s_{[0,j]}}a^{[0,j]}_{k}-\sum_{i=0}^{p}a^{[i,j_{i}]}_{1}.
\end{align*}

The direct computation shows that $\sigma(\mathbf{j})$ is the 
involutive automorphism of $L(P)$ for each $\mathbf{j}\in \mathcal{J}(P)$,
i.e., $\sigma(\mathbf{j})^{2}=\mathrm{id}|_{L(P)}$.

In addition we define the following permutations on $L(P)$. 
For $i_{0}=0,\ldots,p,\, j_{0}=1,\ldots,
r_{i_{0}},\,k_{0}=1,\ldots,s_{[i_{0},j_{0}]}-1$, 
define
\[
\begin{array}{ccccc}
\sigma(i_{0},j_{0},k_{0})\colon &L(P)&\longrightarrow &L(P)&\\
&a^{[i_{0},j_{0}]}_{k_{0}}&\longmapsto&a^{[i_{0},j_{0}]}_{k_{0}+1},&\\
&a^{[i_{0},j_{0}]}_{k_{0}+1}&\longmapsto&a^{[i_0,j_{0}]}_{k_{0}},&\\
&a^{[i,j]}_{k}&\longmapsto&a^{[i,j]}_{k}
&\text{if }(i,j,k)\notin\{(i_{0},j_{0},k_{0}),
(i_{0},j_{0},k_{0}+1)\}.
\end{array}
\]

Then we can define the group $\tilde{W}(P)$ acting on $L(P)$ by
\begin{align*}
&\tilde{W}(P)=\\
&\langle\sigma(\mathbf{j}),\sigma(i,j,k)\mid 
\mathbf{j}\in \mathcal{J}(P),\,i=0,\ldots,p,
\,j=1,\ldots,r_{i},\,k=1,\ldots,s_{[i,j]}-1\rangle.
\end{align*}

Similarly consider the space
$R(P)=
\left\{\left(\alpha^{[i,j]}_{1},\ldots,
\alpha^{[i,j]}_{s_{[i,j]}}\right)
_{\substack{0\le i\le p\\1\le j\le r_{i}}}\,\Big|\, \alpha^{[i,j]}_{k}\in K
\right\}$
with the following transformations 
for $\mathbf{j}=(j_{0},\ldots,j_{p})\in \mathcal{J}(P)$,
\[
\begin{array}{cccc}
    \sigma(\mathbf{j})\colon &R(P)&\longrightarrow &R(P)\\
    &\left(\alpha^{[i,j]}_{1},\ldots,
    \alpha^{[i,j]}_{s_{[i,j]}}\right)
    _{\substack{0\le i\le p\\1\le j\le r_{i}}}
    &\longmapsto
    &\left(\tilde{\alpha}^{[i,j]}_{1},
    \ldots,\tilde{\alpha}^{[i,j]}_{s_{[i,j]}}\right)
    _{\substack{0\le i\le p\\1\le j\le r_{i}}}
\end{array}
\]
where 
\begin{align*}
    &\tilde{\alpha}^{[i,j]}_{k}=\\
    &\begin{cases}\alpha^{[i,j]}_{k}&\text{if }j=j_{i}\text{ and }k=1,\\
    \alpha^{[i,j]}_{k}-(\delta(w^{[i]}_{j}-w^{[i]}_{j_{i}})+1)
    (1-\alpha(\mathbf{j}))&\text{otherwise},
    \end{cases}
\end{align*}
and for $i=0$,
\begin{align*}
    &\tilde{\alpha}^{[0,j]}_{k}=\\
    &\begin{cases}
        \alpha^{[0,j_{0}]}_{1}+2(1-\alpha(\mathbf{j}))
        &\text{if }j=j_{0}\text{ and }k=1,\\
        \alpha^{[0,j]}_{k}
        -(\delta(w^{[0]}_{j}-w^{[0]}_{j_{0}})-1)
        (1-\alpha(\mathbf{j}))&\text{otherwise}.
    \end{cases}
\end{align*}
Here $\alpha(\mathbf{j})=\sum_{i=0}^{p}\alpha^{[i,j_{i}]}_{1}.$
Then it can be seen that $\sigma(\mathbf{j})$ is involutive 
for each $\mathbf{j}\in\mathcal{J}(P)$.
Also we define permutations $\sigma(i_{0},j_{0},s_{0})$ on $R(P)$ 
as we define on  $L(P)$.
Thus we have the action of $\tilde{W}(P)$ on $R(P)$.

Finally we shall consider the space of spectra. 
To so, let us recall the Fuchs relation first.
\begin{prop}[Bertrand \cite{Ber},\cite{BL}]
    Let us consider the above  $P\in W[x]$. Then we have the equation
\begin{align*}
    &\sum_{i=0}^{p}\sum_{j=1}^{r_{i}}
\sum_{k=1}^{s_{[i,j]}}\frac{m^{[i,j]}_{k}
(2\lambda^{[i,j]}_{k}+m^{[i,j]}_{k}-1)}{2}
-\frac{p-1}{2}\mathrm{rank\,}P(\mathrm{rank\,}P-1)
\\
&-\frac{1}{2}\sum_{i=0}^{p}\sum_{1\le j\neq j'\le r_{i}}
\delta(w^{[i]}_{j}-w^{[i]}_{j'})
(\sum_{k=1}^{s_{[i,j]}}m^{[i,j]}_{k})
(\sum_{k'=1}^{s_{[i,j']}}m^{[i,j']}_{k'})\\
&=0.
\end{align*}

This equation called the {\em Fuchs relation} of $P$.
\end{prop}

An analogy of the Fuchs relation
for $\mathbf{a}=(a^{[i,j]}_{1},\ldots,a^{[i,j]}_{s_{[i,j]}})
_{\substack{0\le i\le p\\1\le j\le r_{i}}}\in L(P)$ and
$\overline{\alpha}=(\alpha^{[i,j]}_{1},\ldots,\alpha^{[i,j]}_{s_{[i,j]}})
_{\substack{0\le i\le p\\1\le j\le r_{i}}}\in R(P)$
can be formally considered.
Put 
\begin{align*}
    \Lambda(\mathbf{a};\overline{\alpha})
    =&\sum_{i=0}^{p}\sum_{j=1}^{r_{i}}
\sum_{k=1}^{s_{[i,j]}}\frac{a^{[i,j]}_{k}
(2\alpha^{[i,j]}_{k}+a^{[i,j]}_{k}-1)}{2}-\frac{p-1}{2}\mathrm{rank\,}\mathbf{a}(\mathrm{rank\,}\mathbf{a}-1)\\
&-\frac{1}{2}\sum_{i=0}^{p}\sum_{1\le j\neq j'\le r_{i}}
\delta(w^{[i]}_{j}-w^{[i]}_{j'})
(\sum_{k=1}^{s_{[i,j]}}a^{[i,j]}_{k})
(\sum_{k'=1}^{s_{[i,j']}}a^{[i,j']}_{k'}).
\end{align*}

Then we define the {\em space of spectra},
\[
    S(P)=\{(\mathbf{a},\overline{\alpha})\in L(P)\times R(P)
    \mid \Lambda(\mathbf{a};\overline{\alpha})=0
    \}.
\]

We can define $\tilde{W}(P)$ action on $L(P)\times R(P)$ and show that
this action preserves $S(P)$.
\begin{lem}\label{weylpreserve}
    Let us define $\tilde{W}(P)$ action on $L(P)\times R(P)$ by 
    $\sigma(\mathbf{a},\overline{\alpha})
    =(\sigma(\mathbf{a}),\sigma(\overline{\alpha}))$ 
    for $\sigma\in \tilde{W}(P)$ 
    and $(\mathbf{a},\overline{\alpha})\in L(P) \times R(P)$.
    Then $\tilde{W}(P)$ acts on $S(P)$, namely we have 
    $\tilde{W}(P)S(P)\subset S(P)$.
\end{lem}
We shall prove this lemma in the next subsection.
\subsection{Euler transform and 
Weyl group actions on root lattices}
In the previous subsection, we define the $\tilde{W}(P)$-modules $L(P)$,
$R(P)$ and $S(P)$ as an analogy of the translations of spectra by 
the  Euler transform. 
In this subsection, we shall relate $L(P)$ to a Kac-Moody root lattice and 
show that the $\tilde{W}(P)$ action is compatible with the Weyl group 
action of 
the Kac-Moody root lattice.

Let us define the $\mathbb{Z}$-lattice $Q(P)$ with the basis
\[
    \mathcal{C}=\{c_{\mathbf{j}}\mid \mathbf{j}\in \mathcal{J}(P)\}
            \cup \{c(i,j,k)\mid i=0,\ldots,p,\,j=1,\ldots,r_{i},
            \,k=1,\ldots,s_{[i,j]}-1\},
            \]
namely,
$
Q(P)=\sum_{c\in \mathcal{C}}\Z c.
$
Let us introduce the symmetric bilinear form 
$\langle\ ,\ \rangle$ on $Q(P)$,
\begin{align*}
\langle c_{\mathbf{j}},c_{\mathbf{j}'}\rangle&
=2-\sum_{\substack{0\le i\le p\\ j_{i}\neq j_{i}'}}
\left(\delta(w^{[i]}_{j_{i}}-w^{[i]}_{j'_{i}})+1\right),\\
\langle c_{\mathbf{j}},c(i,j,k)\rangle&=
\begin{cases}
-1&\text{if }j=j_{i}\text{ and }k=1\\
0&\text{otherwise}
\end{cases},\\
\langle c(i,j,k),c(i',j',k')\rangle&=
\begin{cases}
2&\text{if }(i,j,k)=(i',j',k')\\
-1&\text{if }(i,j)=(i',j')\text{ and }|k-k'|=1\\
0&\text{otherwise}
\end{cases}.
\end{align*}
Here $\mathbf{j}=(j_{0},\ldots,j_{p}),
\mathbf{j}'=(j'_{0},\ldots,j'_{p})\in \mathcal{J}(P)$.
We call the lattice $Q(P)$ with the bilinear form $\langle\,,\,\rangle$
the \textit{Kac-Moody root lattice} (see \cite{Kac} for the detail)
associated with $P\in W[x]$.
The {\em simple reflections} with respect to $c\in \mathcal{C}$ is
\[
\sigma_{c}(\alpha)
=\alpha-2\langle c,\alpha\rangle c,
\quad (\alpha\in Q(P)).
\]
The \textit{Weyl group} is the group generated by the simple 
reflections,
i.e., $W(P)=\langle \sigma_{c}\mid c\in \mathcal{C}\, \rangle$.

Then the following theorem shows that 
the lattice $L(P)$ with the group $\tilde{W}(P)$ can be 
seen as a quotient lattice of the Kac-Moody root lattice
$Q(P)$ with the group $W(P)$.
\begin{thm}\label{Weyl group}
Let us define the $\Z$-module homomorphism
\[
\Phi\colon Q(P)\longrightarrow L(P)
\]
as follows.
For 
\[
\alpha=\sum_{\mathbf{j}\in\mathcal{J}(P)}
\alpha_{\mathbf{j}}c_{\mathbf{j}}
+\sum_{i=0}^{p}\sum_{j=1}^{r_{i}}\sum_{k=1}^{s_{[i,j]}-1}
\alpha(i,j,k)c(i,j,k)\in Q(P),
\]
the image 
$\Phi(\alpha)=\left(a^{[i,j]}_{1},\ldots,a^{[i,j]}_{s_{[i,j]}}\right)
_{\substack{0\le i\le p\\1\le j\le r_{i}}}$ is 
\begin{align*}
a^{[i,j]}_{1}&=\sum_{\{\mathbf{j}\in\mathcal{J}(P)
\mid j_{i}=j\}}\alpha_{\mathbf{j}}-\alpha(i,j,1),\\
a^{[i,j]}_{k}&=\alpha(i,j,k-1)-\alpha(i,j,k)
\quad \text{for }2\le k\le s_{[i,j]}.
\end{align*}
Here we put $\alpha(i,j,s_{[i,j]})=0$.
                Then we have the following.
\begin{enumerate}
\item The map $\Phi$ is surjective.
\item The map $\Phi$ is injective if and only if 
      $
      \#\{r_{i}\mid r_{i}> 1,\,i=0,\ldots,p\}\le 1.
      $
\item The Weyl group action on $Q(P)$ is compatible with  
    the action of $\tilde{W}(P)$ on $L(P)$. 
    Namely, we have 
    \begin{align*}
        \Phi(\sigma_{c_{\mathbf{j}}}(\alpha))&
        =\sigma(\mathbf{j})(\Phi(\alpha)),
        \quad (\, \mathbf{j}\in \mathcal{J}(P)),\\
    \Phi(\sigma_{c(i,j,k)}(\alpha))&=\sigma(i,j,k)(\Phi(\alpha)),
    \quad (i=0,\ldots,p,\,j=1,\ldots,r_{i},\,k=1,\ldots,s_{[i,j]}-1),
    \end{align*}
    for all $\alpha\in Q(P)$.
\item If $\alpha\in \mathrm{Ker\,}\Phi$, 
    then $\langle \alpha,\beta\rangle =0$ for any $\beta\in Q(P)$.
\end{enumerate}
\end{thm}
\begin{proof}        
    Let us take an arbitrary element 
    $\hat{\mathbf{j}}=(\hat{j}_{0},\ldots,\hat{j}_{p})
    \in \mathcal{J}(P)$. 
    Then we can check that images of 
     \[\begin{cases}c(i,j,s) 
             &\text{for }i=0,\ldots, p,\,
             j=1,\ldots,k_{i},\,s=1,\ldots,l_{i,j}-1,\\
             c_{\mathbf{j}} 
             &\text{for }\mathbf{j}\in
             \bigcup_{i=0}^{p}
             \{(f^{(i)}_{0},\ldots,f^{(i)}_{p})\in \mathcal{J}(P)
             \mid f^{(i)}_{i'}= \hat{j}_{i'}\text
             { if }i'\neq i\}
     \end{cases}
     \]
     generate $L(P)$. Hence $\Phi$ is surjective.

     Let us show 2. Since ranks of free $\Z$-modules $Q(P)$ 
     and $L(P)$ are
     \begin{align*}
         \mathrm{rank}_{\Z\text{-mod}}Q(P)&
         =\prod_{i=0}^{p}r_{i}+\sum_{i=0}^{p}
         \sum_{j=1}^{r_{i}}(s_{[i,j]}-1),\\
         \mathrm{rank}_{\Z\text{-mod}}L(P)
         &=\sum_{i=0}^{p}\sum_{j=1}^{r_{i}}s_{[i,j]}-p
     \end{align*}
     respectively, thus $\mathrm{rank}_{\Z\text{-mod}}Q(P)
     -\mathrm{rank}_{\Z\text{-mod}}L(P)
     =\prod_{i=0}^{p}r_{i}-\sum_{i=0}^{p}r_{i}+p.$
     Here we notice that 
     $
        \prod_{i=0}^{p}r_{i}-\sum_{i=0}^{p}r_{i}+p\ge 0.
     $
     Indeed $g(r_{0},\ldots,r_{p})
     =\prod_{i=0}^{p}r_{i}-\sum_{i=0}^{p}r_{i}$
     is the increasing function of each $r_{i}\in \mathbb{Z}_{>0}$
     for $i=0,\ldots,p$,
     since $\frac{\partial}{\partial r_{i}}
     g(r_{0},\ldots,r_{p})
     =\prod_{j\neq i}r_{j}-1\ge 0.$
     Thus, 
     \[
     g(r_{0},\ldots,r_{p})\ge g(1,\ldots,1)=-p.
     \]
     If we assume that there exist at least two $r_{i_{1}}$ 
     and $r_{i_{2}}$ satisfying $r_{i_{1}}\ge 2$ 
     and $r_{i_{2}}\ge 2$, then
     \[
     g(r_{0},\ldots,r_{p})\ge 4-(2+2+(p-1))=-p+1.
     \]
     Hence $\mathrm{rank}_{\Z\text{-mod}}Q(P)
     -\mathrm{rank}_{\Z\text{-mod}}L(P)\ge 1.$
     This shows $\Phi$ is not injective.
     On the contrary, if all $r_{i}$ are $r_{i}=1$
     except only one $r_{i_{0}}$, then
     $
     g(r_{0},\ldots,r_{p})=r_{i_{0}}-r_{i_{0}}-p.
     $
     Hence 
     $
     \mathrm{rank}_{\Z\text{-mod}}Q(P)
     =\mathrm{rank}_{\Z\text{-mod}}L(P).
     $
    Then since $\Phi$ is surjective, this shows that $\Phi$ is injective.

    Let us show 3. We have
    \begin{align*}
        &d(\Phi(\alpha);\mathbf{j})
        =\sum_{i=1}^{p}\sum_{j=1}^{r_{i}}
        (\delta(w^{[i]}_{j}-w^{[i]}_{j_{i}})+1)
        \sum_{k=1}^{s_{[i,j]}}a^{[i,j]}_{k}\\
        &\quad +\sum_{j=1}^{r_{0}}(\delta
        (w^{[0]}_{j}-w^{[0]}_{j_{0}})-1)
        \sum_{k=1}^{s_{0,j}}a^{[0,j]}_{k}
        -\sum_{i=0}^{p}a^{[i,j_{i}]}_{1}\\
        &=\sum_{i=1}^{p}\sum_{l=0}^{\infty}(l+1)
        \sum_{\{\mathbf{j}'\in \mathcal{J}(P)\mid
        \delta(w^{[i]}_{j'_{i}}-w^{[i]}_{j_{i}})=l\}}
        \alpha_{\mathbf{j}'}
        +\sum_{l=0}^{\infty}(l-1)
        \sum_{\{\mathbf{j}''\in\mathcal{J}(P)
        \mid \delta(w^{[0]}_{j''_{0}}-w^{[0]}_{j_{0}})=l\}}
        \alpha_{\mathbf{j}''}\\
        &\quad-\sum_{i=0}^{p}\sum_{\{\mathbf{j}'''\in\mathcal{J}(P)
        \mid j'''_{i}=j_{i}\}}\alpha_{\mathbf{j}'''}
        +\sum_{i=0}^{p}\alpha(i,j_{i},1)\\ &=\sum_{\mathbf{j}'\in\mathcal{J}(P)}\alpha_{\mathbf{j}'}
        \bigg(\sum_{i=0}^{p}
        \big(\delta(w^{[i]}_{j'_{i}}-w^{[i]}_{j_{i}})+1\big)-2
        -\#\{k\mid j'_{k}=j_{k}, k=0,\ldots,p\}\bigg)\\
        &\quad+\sum_{i=0}^{p}\alpha(i,j_{i},1)\\
        &=-\langle c_{\mathbf{j}},\alpha\rangle.
      \end{align*}
      Hence we have 
      $\Phi(\sigma_{c_{\mathbf{j}}}(\alpha))
      =\sigma(\mathbf{j})(\Phi(\alpha)).$
      Equations $\Phi(\sigma_{c(i,j,k)}(\alpha))
      =\sigma(i,j,k)(\Phi(\alpha))$
   similarly follow.  
   
   Let us show 4. If $\alpha\in \Ker\Phi$, 
   then $d(\Phi(\alpha);\mathbf{j})=0$. 
   Thus $\langle c_{\mathbf{j}},\alpha\rangle=0$ 
   for all $\mathbf{j}\in \mathcal{J}(P)$.
   Similarly we have $\langle c(i,j,k),\alpha\rangle=0$ 
   for all $i=0,\ldots,p,\,j=1,\ldots,r_{i},\,k=1,\ldots,s_{[i,j]}-1$. 
   Hence if $\alpha\in \Ker\Phi$, 
   then $\langle \beta,\alpha\rangle=0$ for all $\beta\in Q(P)$. 
\end{proof}

The Weyl group $W(P)$ can be seen as a Coxeter group defined by generators 
 and relations. Let us define integers
 \[
     m(c,c')=
     \begin{cases}
         2&\text{if }\langle c,c'\rangle = 0,\\
         3&\text{if }\langle c,c'\rangle = -1,\\
         \infty&\text{otherwise}
     \end{cases}
 \]
 for $c,c'\in \mathcal{C}$.
 Then $W(P)$ can be seen as the group generated by following generators 
 with relations,
 \begin{align*}
     W(P)=\langle
     \sigma_{c}\ (c\in \mathcal{C})\mid \, \sigma_{c}^{2}=e,\,
     (\sigma_{c}\sigma_{c'})^{m(c,c')}=e,\,c,c'\in \mathcal{C}\,
     \rangle
 \end{align*}
 (see Proposition 3.13 in \cite{Kac}). 
 Here $e$ denotes the identity element in $W(P)$ and use the notation $\sigma^{\infty}=e$ for $\sigma\in W(P)$.
 \begin{prop}\label{Weylexp}
     The Weyl group $W(P)$ acts on $R(P)$ as follows,
     \begin{align*}
         \sigma_{c_{\mathbf{j}}}(\mathbf{\mu})
         &=\sigma(\mathbf{j})(\mathbf{\mu}),\\
     \sigma_{c(i,j,k)}(\mathbf{\mu})&=\sigma(i,j,k)(\mathbf{\mu})
     ,\quad \mu\in R(P). 
     \end{align*}
 \end{prop}
\begin{proof}
   We need to see that the $W(P)$ action on $R(P)$ is well-defined.
   Namely, if for $c,c'\in \mathcal{C}$
   reflections $\sigma_{c},\sigma_{c'}$ 
   satisfy Coxeter relations
   \begin{align*}
       &\sigma_{c}^{2}=\sigma_{c'}^{2}=e,\\
       &(\sigma_{c}\sigma_{c'})^{m(c,c')}=e
   \end{align*}
   in $W(P)$ for the positive integer $m(c,c')$
   (sometimes it is $\infty$), then we should have
  \begin{align}
       &\sigma_{c}^{2}\mu=\sigma_{c'}^{2}\mu=\mu,\label{involutive}\\
       &(\sigma_{c}\sigma_{c'})^{m(c,c')}\mu=\mu\label{Coxeter}
   \end{align}
   for all $\mu\in R(P)$. 
   The involutive relations $(\ref{involutive})$ 
   are directly from the definition. 
   We check the relations $(\ref{Coxeter})$.

   Let us take $\mathbf{j},\mathbf{j}'\in \mathcal{J}(P)\, 
   (\mathbf{j}\neq \mathbf{j}')$ 
   and compute $(\sigma(\mathbf{j})\sigma(\mathbf{j}))^{m}$ on $R(P)$. 
   For $\nu=\left(\nu^{[i,j]}_{1},\ldots,\nu^{[i,j]}_{s_{[i,j]}}\right)
   _{\substack{0\le i\le p\\1\le j\le r_{i}}}\in R(P)$,
   we can see $$(\sigma(\mathbf{j}')\sigma(\mathbf{j}))^{m}\nu
   =\nu(m)=\left(\nu(m)^{[i,j]}_{1},\ldots,
   \nu(m)^{[i,j]}_{s_{[i,j]}}\right)
   _{\substack{0\le i\le p\\1\le j\le r_{i}}}$$ are as follows.
   For  $i=1,\ldots,p$, we have the following.
   \begin{itemize}
       \item If $j_{i}=j'_{i}$,
           \[
               \nu(m)^{[i,j]}_{k}=\begin{cases}
                   \nu^{[i,j]}_{1},&\text{if }j=j_{i}=j'_{i}\text{ and }
                   k=1,\\
               \begin{split}\nu^{[i,j]}_{k}&-(\delta
                   (w^{[i]}_{j}-w^{[i]}_{j'_{i}})+1)
                   \sum_{u=1}^{m}\mu(u)_{\mathbf{j}'}\\
                   &-(\delta(w^{[i]}_{j}-w^{[i]}_{j_{i}})+1)
                   \sum_{u=1}^{m}\mu(u)_{\mathbf{j}},\end{split}
                   &\text{otherwise}.
           \end{cases}
           \]
       \item If $j_{i}\neq j'_{i}$,
           \[
               \nu(m)^{[i,j]}_{k}=\begin{cases} 
                   \nu^{[i,j]}_{1}-(\delta
                   (w^{[i]}_{j_{i}}-w^{[i]}_{j'_{i}})+1)
                   \sum_{u=1}^{m}\mu(u)_{\mathbf{j}'},
                   &\text{if }j=j_{i}\text{ and }k=1,\\
                   \nu^{[i,j]}_{1}-(\delta
                   (w^{[i]}_{j'_{i}}-w^{[i]}_{j_{i}})+1)
                   \sum_{u=1}^{m}\mu(u)_{\mathbf{j}},
                   &\text{if }j=j'_{i}\text{ and }k=1,\\
                   \begin{split}\nu^{[i,j]}_{k}&
                   -(\delta(w^{[i]}_{j}-w^{[i]}_{j'_{i}})+1)
                   \sum_{u=1}^{m}\mu(u)_{\mathbf{j}'}\\
                   &-(\delta(w^{[i]}_{j}-w^{[i]}_{j_{i}})+1)
                   \sum_{u=1}^{m}\mu(u)_{\mathbf{j}},\end{split}
                   &\text{otherwise}.
           \end{cases}
           \]
       \end{itemize}
           Similarly $\nu(m)^{[0,j]}_{k}$ are as follows.
 \begin{itemize}
       \item If $j_{0}=j'_{0}$,
           \[
               \nu(m)^{[0,j]}_{k}=\begin{cases}
                   \nu^{[0,j]}_{1}+2(\sum_{u=1}^{m}(\mu(u)_{\mathbf{j}}
                   +\mu(u)_{\mathbf{j}})),
                   &\text{if }j=j_{0}=j'_{0}\text{ and }k=1,\\
                   \begin{split}\nu^{[0,j]}_{k}&-(\delta
                       (w^{[0]}_{j}-w^{[0]}_{j'_{0}})-1)
                       \sum_{u=1}^{m}\mu(u)_{\mathbf{j}'}\\
                       &-(\delta(w^{[0]}_{j}-w^{[0]}_{j_{0}})-1)
                       \sum_{u=1}^{m}\mu(u)_{\mathbf{j}},\end{split}
                       &\text{otherwise}.
           \end{cases}
           \]
       \item If $j_{0}\neq j'_{0}$,
           \[
               \nu(m)^{[0,j]}_{k}=\begin{cases} 
                   \begin{split}\nu^{[0,j]}_{1}
                       &+2\sum_{u=1}^{m}\mu(u)_{\mathbf{j}}\\
                       &-(\delta
                       (w^{[0]}_{j_{0}}-w^{[0]}_{j'_{0}})-1)
                       \sum_{u=1}^{m}\mu(u)_{\mathbf{j}'},\end{split}&
                       \text{if }j=j_{0},\text{ and }k=1,\\
                       \begin{split}\nu^{[0,j]}_{1}&
                           +2\sum_{u=1}^{m}\mu(u)_{\mathbf{j}'}\\
                           &-(\delta(w^{[0]}_{j'_{0}}
                           -w^{[0]}_{j_{0}})-1)
                           \sum_{u=1}^{m}\mu(u)_{\mathbf{j}},\end{split}
                           &\text{if }j=j'_{0}\text{ and }k=1,\\
                           \begin{split}\nu^{[0,j]}_{k}
                               &-(\delta
                               (w^{[0]}_{j}-w^{[0]}_{j'_{0}})-1)
                               \sum_{u=1}^{m}\mu(u)_{\mathbf{j}'}\\
                               &-(\delta
                               (w^{[0]}_{j}-w^{[0]}_{j_{0}})-1)
                               \sum_{u=1}^{m}\mu(u)_{\mathbf{j}},
                           \end{split}&\text{otherwise}.
                       \end{cases}
           \]
           Here $\mu(u)_{\mathbf{j}}$ and $\mu(u)_{\mathbf{j}'}$
           are defined by
           \begin{align*}
               \mu(1)_{\mathbf{j}}&=1-\sum_{i=0}^{p}\nu^{[i,j_{i}]}_{1},\\
               \mu(u)_{\mathbf{j}}&=-\mu(u-1)_{\mathbf{j}}
               +E\mu(u-1)_{\mathbf{j}'},\\
               \mu(1)_{\mathbf{j}'}&=1-\sum_{i=0}^{p}\nu^{[i,j'_{i}]}_{1}
               +E\mu(1)_{\mathbf{j}},\\
               \mu(u)_{\mathbf{j}'}&=-\mu(u-1)_{\mathbf{j}'}
               +E\mu(u)_{\mathbf{j}},
           \end{align*}
           where
           $E=
               \sum_
               {\substack{0\le i\le p\\ j_{i}\neq j'_{i}}}\left(\delta
               (w^{[i]}_{j_{i}}-w^{[i]}_{j'_{i}})+1\right).
            $
       \end{itemize}
       Then it follows that 
       \begin{align*}
           \sum_{u=1}^{m}\mu(u)_{\mathbf{j}}&=
           \begin{cases}
               E_{m}\mu(m-1)_{\mathbf{j}}&\text{ if $m$ is odd},\\
               E_{m}\mu(m-1)_{\mathbf{j}'}&\text{ if $m$ is even},
           \end{cases}\\
           \sum_{u=1}^{m}\mu(u)_{\mathbf{j}'}&=
           \begin{cases}
               E_{m}\mu(m-1)_{\mathbf{j}'}&\text{ if $m$ is odd},\\
               E_{m}\mu(m-1)_{\mathbf{j}}&\text{ if $m$ is even}.
           \end{cases}
       \end{align*}
       Here each $E_{m}$ is inductively defined by
       \[
           E_{m}=-E_{m-2}+E\cdot E_{m-1},
       \]
       from the initial values $E_{0}=1, E_{1}=E$. 
       Then we have
       \[
           \begin{cases} E_{2}=0&\text{ if $E=0$},\\
               E_{2}\neq 0\text{ and } E_{3}=0&\text{ if $E=1$}.
           \end{cases}
       \]
       If $E=2$, then the recurrence relation implies 
       \[
           E_{m}=m+1.
       \]
       If $E> 2$, we have
       \[
           E_{m}=\frac{\beta^{m+1}-\alpha^{m+1}}{\beta-\alpha},
       \]
       where $\alpha=\frac{E+\sqrt{E^{2}-4}}{2}$ and 
       $\beta=\frac{E-\sqrt{E^{2}-4}}{2}$. 
       Thus $E_{m}\neq 0$ for all $m\ge 0$ if $E\ge 2$.

Then we have 
\[(\sigma(\mathbf{j})\sigma(\mathbf{j}'))^{m}
    =\mathrm{id}|_{R(P)}\ \text{for}
\begin{cases}m=2&\text{if }E=0,\\
    m=3&\text{if }E=1.\\
\end{cases}
\]
Similarly the direct computation shows that
\[
    (\sigma(\mathbf{j})\sigma(i,j,k))^{m}
    =\mathrm{id}|_{R(P)}\ \text{for }
\begin{cases}
    m=3&\text{if }j=j_{i}\text{ and }k=1,\\
    m=2&\text{otherwise},
\end{cases}
\]
and 
\[
(\sigma(i,j,s)\sigma(i',j',s'))^{m}=\mathrm{id}|_{R(P)}\ \text{for }
\begin{cases}
    m=3&\text{if }(i,j)=(i',j')\text{ and }|s-s'|=1,\\
    m=2&\text{otherwise}.
\end{cases}
\]

\end{proof}

For $\mathbf{a}\in L(P)$ let us define 
$$\mathrm{idx\,}\mathbf{a}
=\langle \Phi^{-1}(\mathbf{a}),\Phi^{-1}(\mathbf{a})\rangle$$
and call the \textit{index of rigidity} of $\mathbf{a}$.
Theorem \ref{Weyl group} assures that this definition is well-defined
and moreover $\mathrm{idx\,}\mathbf{a}$ is invariant under the action of 
$\tilde{W}(P)$ since we know $\langle w(\alpha),w(\beta)\rangle=
\langle \alpha,\beta\rangle$ for any $\alpha,\beta\in Q(P)$
and $w\in W(P)$.

For
    $\mathbf{a}=\left(a^{[i,j]}_{1},\ldots,
    a^{[i,j]}_{s_{[i,j]}}\right)_
    {\substack{0\le i\le p\\1\le j\le r_{i}}}\in L(P)$
    and 
    $\overline{\alpha}
    =\left(\alpha^{[i,j]}_{1},\ldots,
    \alpha^{[i,j]}_{s_{[i,j]}}\right)
    _{\substack{0\le i\le p\\1\le j\le r_{i}}}\in R(P)$,
    let us define
    $\overline{\alpha}\cdot \mathbf{a}=
    \sum_{i=0}^{p}\sum_{j=1}^{r_{i}}
    \sum_{k=1}^{s_{[i,j]}}\alpha^{[i,j]}_{k}a^{[i,j]}_{k}$.

Then we can rewrite the Fuchs relation in terms of the index of rigidity.
\begin{lem}\label{fuchsindex}
    For $(\mathbf{a},\overline{\alpha})
    \in S(P)$,
    we have the equations
    \begin{align*}
        &\mathrm{idx\,}\mathbf{a}
        =-\sum_{i=0}^{p}\sum_{1\le j\neq j'\le r_{i}}
        \delta(w^{[i]}_{j}-w^{[i]}_{j'})
        \left(\sum_{k=1}^{s_{[i,j]}}a^{[i,j]}_{k}\right)
        \left(\sum_{k=1}^{s_{[i,j']}}a^{[i,j']}_{k}\right)\\
        &\quad\quad\quad
        +\sum_{i=0}^{p}\sum_{j=1}^{r_{i}}\sum_{k=1}^{s_{[i,j]}}
        (a^{[i,j]}_{k})^{2}-(p-1)(\mathrm{rank\,}\mathbf{a})^{2},\\
        &\Lambda(\mathbf{a},\overline{\alpha})=
        \overline{\alpha}\cdot \mathbf{a}
        +\frac{\mathrm{idx\,}\mathbf{a}}{2}
        -\mathrm{rank\,}\mathbf{a}.
    \end{align*}
\end{lem}
\begin{proof}
    Define $\mathcal{J}(P)^{(i)}_{1}=\{(j^{(i)}_{0},
    \ldots,j^{(i)}_{p})\in \mathcal{J}(P)\mid 
    j_{i'}^{(i)}=1\text{ if }i\neq i'\}$ and 
    $\mathcal{J}(P)_{1}=\amalg_{i=0}^{p}\mathcal{J}(P)^{(i)}_{1}$,
    the disjoint union of $\mathcal{J}(P)^{(i)}_{1}$, $i=0,\ldots,p$. 
    For $i=0,\ldots,p,\,j=1,\ldots,r_{i}$ let us define
    $\mathbf{j}^{(i,j)}=(j_{0},\ldots,j_{p})\in \mathcal{J}(P)^{(i)}_{1}$
    by $j_{i}=j$ and $j_{i'}=1$ for $i'\neq i$. 

    Then choose the element 
    $\xi=\sum_{c\in \mathcal{C}}\xi_{c}c \in \Phi^{-1}(\mathbf{a})$
    so that
    \begin{align*}
        \xi_{c_{\mathbf{j}}}&=\begin{cases}
            0&\text{if }\mathbf{j}\notin 
            \mathcal{J}(P)_{1},\\
            \sum_{k=1}^{s_{[i,j]}}a^{[i,j]}_{k}
            &\text{ if }\mathbf{j}=\mathbf{j}^{(i,j)}\text{ and }
            j\neq 1,\\
            \sum_{i=0}^{p}\sum_{k=1}^{s_{[i,1]}}a^{[i,1]}_{k}
            -p\,\mathrm{rank\,}\mathbf{a}&
            \text{ if }\mathbf{j}=\mathbf{1}=(1,1,\ldots,1),
        \end{cases}\\
        \xi_{c(i,j,k)}&=\sum_{k'=k+1}^{s_{[i,j]}}a^{[i,j]}_{k'}.
    \end{align*}
    For $i=0,\ldots,p,\,j=1,\ldots,r_{i}$ let us define
    $\mathbf{j}^{(i,j)}=(j_{0},\ldots,j_{p})\in \mathcal{J}(P)^{(i)}_{1}$
    by $j_{i}=j$ and $j_{i'}=1$ for $i'\neq i$. 
    Then 
    \begin{align*}
        \mathrm{idx\,}\mathbf{a}&=\langle\xi,\xi\rangle\\
        &=\sum_{\mathbf{j}\in \mathcal{J}(P)}\xi_{c_{\mathbf{j}}}
        \langle c_{\mathbf{j}},\xi\rangle+
        \sum_{i=0}^{p}\sum_{j=1}^{r_{i}}\sum_{k=1}^{s_{[i,j]-1}}
        \xi_{c(i,j,k)}\langle c(i,j,k),\xi\rangle\\
        &=-\sum_{\mathbf{j}\in \mathcal{J}(P)}\xi_{c_{\mathbf{j}}}
        d(\mathbf{a};\mathbf{j})\\
        &\quad+
        \sum_{i=0}^{p}\sum_{j=1}^{r_{i}}\sum_{k=1}^{s_{[i,j]-1}}
        \xi_{c(i,j,k)}
        \left(2\xi_{c(i,j,k)}-\left(\xi_{c(i,j,k-1)}+\xi_{c(i,j,k+1)}
        \right)\right).
    \end{align*}
    Here we put $\xi_{c(i,j,s_{[i,j]})}=0$, $\xi_{c(i,j,0)}
    =\xi_{c_{\mathbf{j}^{(i,j)}}}$ if $j\neq 1$  and 
    $\xi_{c(i,1,0)}=
    \xi_{c_{\mathbf{1}}}$.

    Then 
    \begin{align*}
        &\mathrm{idx\,}\mathbf{a}=
        -\sum_{i=0}^{p}\sum_{1\le j\neq j'\le r_{i}}
        \delta(w^{[i]}_{j}-w^{[i]}_{j'})
        \left(\sum_{k=1}^{s_{[i,j]}}a^{[i,j]}_{k}\right)
        \left(\sum_{k=1}^{s_{[i,j']}}a^{[i,j']}_{k}\right)\\
        &+\sum_{i=0}^{p}\sum_{j=1}^{r_{i}}\left(
        \sum_{k=1}^{s_{[i,j]}}a_{k}^{[i,j]}\right)a^{[i,j]}_{1}
        -(p-1)(\mathrm{rank\,}\mathbf{a})^{2}\\
        &+ \sum_{i=0}^{p}\sum_{j=1}^{r_{i}}\sum_{k=1}^{s_{[i,j]-1}}
        \xi_{c(i,j,k)}
        (2\xi_{c(i,j,k)}-(\xi_{c(i,j,k-1)}+\xi_{c(i,j,k+1)}))\\
        &=-\sum_{i=0}^{p}\sum_{1\le j\neq j'\le r_{i}}
        \delta(w^{[i]}_{j}-w^{[i]}_{j'})
        \left(\sum_{k=1}^{s_{[i,j]}}a^{[i,j]}_{k}\right)
        \left(\sum_{k=1}^{s_{[i,j']}}a^{[i,j']}_{k}\right)\\
        &+\sum_{i=0}^{p}\sum_{j=1}^{r_{i}}\left(
        \sum_{k=1}^{s_{[i,j]}}a_{k}^{[i,j]}\right)a^{[i,j]}_{1}
        -(p-1)(\mathrm{rank\,}\mathbf{a})^{2}\\
        &+ \sum_{i=0}^{p}\sum_{j=1}^{r_{i}}\left(\sum_{k=1}^{s_{[i,j]}}
        (a^{[i,j]}_{k})^{2}
        -a^{[i,j]}_{1}\sum_{k'=1}^{s_{[i,j]}}a^{[i,j]}_{k'}\right)
    \end{align*}
    \begin{align*}
        &=-\sum_{i=0}^{p}\sum_{1\le j\neq j'\le r_{i}}
        \delta(w^{[i]}_{j}-w^{[i]}_{j'})
        \left(\sum_{k=1}^{s_{[i,j]}}a^{[i,j]}_{k}\right)
        \left(\sum_{k=1}^{s_{[i,j']}}a^{[i,j']}_{k}\right)\\
        &+\sum_{i=0}^{p}\sum_{j=1}^{r_{i}}\sum_{k=1}^{s_{[i,j]}}
        (a^{[i,j]}_{k})^{2}-(p-1)(\mathrm{rank\,}\mathbf{a})^{2}.
    \end{align*}

    Then the second equation directly follows from the first one.
\end{proof}

By using the description of $\Lambda(\mathrm{a},\overline{\alpha})$ in 
Lemma \ref{fuchsindex}, let us give the proof of Lemma \ref{weylpreserve}.
\begin{proof}[Proof of Lemma \ref{weylpreserve}]
        One can directly check that          
        $\sigma(\mathbf{j})(\overline{\alpha})\cdot 
        \sigma(\mathbf{j})(\mathbf{a})=\overline{\alpha}\cdot\mathbf{a}
        +d(\mathbf{a};\mathbf{j})$ for all $\mathbf{j}\in \mathcal{J}(P)$
        (cf. Theorem 5.2 in \cite{O}).
        Note that $\sigma(i,j,k)$ does not change $\overline{\alpha}\cdot
        \mathbf{a}$ and $\mathrm{rank\,}\mathbf{a}$.
        Therefore since  we have the equation
        \[\Lambda(\mathbf{a},\overline{\alpha})=
        \overline{\alpha}\cdot \mathbf{a}
        +\frac{\mathrm{idx\,}\mathbf{a}}{2}
        -\mathrm{rank\,}\mathbf{a},
        \]
        and $\mathrm{idx\,}\mathbf{a}$ is invariant under the 
        $\tilde{W}(P)$ action,
        we have $\Lambda(\sigma(\overline{\alpha},\mathbf{a}))=
        \Lambda(\overline{\alpha},\mathbf{a})$.
\end{proof}
  \subsection{Examples : affine Weyl group symmetries of Heun equations.}\label{Heun and Painleve}
  Let us see some examples of Theorem \ref{Weyl group}. 
  As examples, we consider the Heun differential operator 
  and its confluent operators (cf. \cite{S-L} for instance). 
  
  \vspace{3mm}
  \noindent (1) Heun differential operator.

 The Heun differential operator is the differential operator of the form
 \begin{multline*}
 P=x(x-1)(x-t)\partial^{2}+\{c(x-1)(x-t)+dx(x-t)\\
 +(a+b+1-c-d)x(x-1)\}\partial+(abx-\lambda).
 \end{multline*}

 This has regular singular points at $x=0,1,t,\infty$ 
 and the following spectra,

 \begin{align*}
     &\{(0,1-c);(1,1)\}\quad\text{at $x=0$},\\
     &\{(0,1-d);(1,1)\}\quad\text{at $x=1$},\\
     &\{(0,c+d-a-b);(1,1)\}\quad\text{at $x=t$},\\
     &\{(a,b);(1,1)\}\quad\text{at $x=\infty$}.
 \end{align*}

 By Theorem \ref{Weyl group}, we can define the root lattice $Q(P)$ 
 with the following Dynkin diagram.
 
 
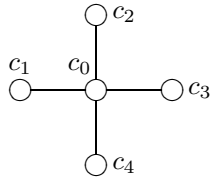
\hspace{40mm}
\begin{xy}
    \ar@{-} *++!D{c_{1}} *\cir<4pt>{};
    (10,0) *++!D!R(0.6){c_{0}} *\cir<4pt>{}="A",
    \ar@{-} "A"; (10,10) *++!L{c_{2}} *\cir<4pt>{},
    \ar@{-} "A"; (20,0) *++!L{c_{3}} *\cir<4pt>{},
    \ar@{-} "A"; (10,-10) *++!L{c_{4}} *\cir<4pt>{}
 \end{xy}

 Note that $\mathcal{J}(P)$ consists of a point. Thus here we denote by
 $c_{0}$ the corresponding base of $Q(P)$.
 Also $c_{i}$ for $i=1,\ldots,4$ denote the basis of $Q(P)$ corresponding 
 to $c(i,1,1)$ in the previous notation.
 Here the diagram is drawn by the following rule. 
 If $c_{i}$ and $c_{j}$ in  the basis of $Q(P)$ 
 satisfy $\langle c_{i},c_{j}\rangle=-m(i,j)$, 
 then corresponding vertices $c_{i}$ and $c_{j}$ 
 are connected by $m(i,j)$ edges. We can see 
 \[
 \langle c_{1},c_{0}\rangle =-1,\quad \langle c_{1},c_{2}\rangle=0
 \]
 from this diagram for example.
 
 This diagram is that of the affine $D_{4}^{(1)}$ type root system. 
 For the above spectra defines the element
 \[
 \mathbf{m}(P)=\left((1,1),(1,1),(1,1),(1,1)\right)\in L(P)
 \]
and  we can associate the element in $Q(P)$
 \[
 \Phi^{-1}(\mathbf{m}(P))=2c_{0}+\sum_{i=1}^{4}c_{i}.
 \]
 This is an imaginary root of $Q(P)$. 
 
 We can see that $\delta(P)=\Phi^{-1}(\mathbf{m}(P))$ is fixed under
 the action of $W(P)$. 
 Namely, Euler transforms $E(\mathbf{j})$
 and permutations $\sigma(i,j,s)$
 do not change the spectral type $\mathbf{m}(P)$. 
 On the other hand, characteristic exponents are changed 
 by $E(\mathbf{j})$ and permutations. 
 As we see in Proposition $\ref{Weylexp}$, 
 the Weyl group $W(P)$ acts on the space 
 of characteristic exponents $R(P)$ as well. 
 Thus we can conclude that characteristic 
 exponents of the Heun differential operator 
 has affine $D^{(1)}_{4}$ Weyl group symmetry 
 generated by twisted Euler transform and permutations. 

 \vspace{3mm}
 \noindent (2) Confluent Heun differential operator.

 The confluent Heun differential operator is
 \[
 P^{c}= x(x-1)\partial^{2}+\{-tx(x-1)+c(x-1)+dx\}\partial
     +(-tax+\lambda).
 \]
 This operator has regular singular points at $x=0,1$ and irregular singular point at $x=\infty$. The spectra are
 \begin{align*}
     \{(0,1-c);(1,1)\}\quad\text{at $x=0$},\\
     \{(0,1-d);(1,1)\}\quad\text{at $x=1$},
 \end{align*}
 for regular singular points and 
 \begin{align*}
 \{(a);(1)\}\quad \text{with $w_{1}=0$},\\
 \{(c+d-a);(1)\}\quad \text{with $w_{2}=tx$},
 \end{align*}
 for the irregular singular point $x=\infty$. Then the corresponding root system has the following extended Dynkin diagram.
 
 
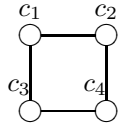
\hspace{40mm}
 \begin{xy}  
     \ar@{-}(0,5) *++!D{c_{1}} *\cir<4pt>{}="B"; (10,5) *++!D{c_{2}} *\cir<4pt>{}="C",
     \ar@{-}(0,-5) *++!D!R(0.4){c_{3}} *\cir<4pt>{}="D"; (10,-5) *++!D!R(0.4){c_{4}} *\cir<4pt>{}="E",
     \ar@{-}"B";"D",
     \ar@{-}"C";"E"
\end{xy}

This corresponds to  the affine $A^{(1)}_{4}$ root system. And we have 
\[
\delta(P^{c})=\Phi^{-1}(\mathbf{m}(P^{c}))=\sum_{i=1}^{4}c_{i}.
\]
This $\delta(P^{c})$ is the imaginary root of $Q(P^{c})$ 
and $W(P^{c})$-invariant. 
Hence as well as the Heun differential operator, 
we can conclude that the characteristic exponents 
of confluent Heun differential operator has affine $A^{(1)}_{4}$
Weyl group symmetry generated 
by twisted Euler transforms and permutations.

\vspace{3mm}
\noindent (3) Biconfluent Heun differential operator.

Let us consider the biconfluent Heun differential operator,
\[
P^{bc}=x\partial^{2}+(-x^{2}-tx+c)\partial+(-ax+\lambda).
\]

This has regular singular point at $x=0$ with the spectrum,
\[
\{(0,1-c);(1,1)\},
\]
and irregular singular point at $x=\infty$ with the spectra,
\begin{align*}
    &\{(a);(1)\}\quad\text{with $w_{1}=0$},\\
    &\{(c+1-a);(1)\}\quad\text{with $w_{2}=x+t$}.
\end{align*}

The corresponding diagram and the element in $Q(P^{bc})$ are as follows,

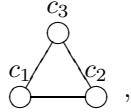
\hspace{40mm}
\begin{xy}
    \ar@{-}(0,0) *++!D{c_{1}} *\cir<4pt>{}="F"; (10,0) *++!D{c_{2}} *\cir<4pt>{}="G",
    \ar@{-} (5,8.5) *++!D{c_{3}} *\cir<4pt>{}="H"; "F",
     \ar@{-}"H";"G"
\end{xy},
\[
\delta(P^{bc})=\Phi^{-1}(\mathbf{m}(P^{bc}))=\sum_{i=1}^{3}c_{i}.
\]

Hence this is the affine $A^{(1)}_{3}$ root system 
and $\delta(P^{bc})$ is the imaginary root of this root system.
As well as the above examples, we can see that $P^{bc}$ 
has the affine $A^{(1)}_{3}$ Weyl group symmetry generated 
by twisted Euler transforms and permutations.

\vspace{3mm}
\noindent (4) Triconfluent Heun differential operator.

The triconfluent Heun differential operator is 
\[
P^{tc}=\partial^{2}+(-x^{2}-t)\partial+(-ax+\lambda).
\]
As well as the above examples, 
we can see that $P^{tc}$ has 
the affine $A^{(1)}_{2}$ Weyl group symmetry generated 
by twisted Euler transforms. 

Indeed the spectra are
\begin{align*}
    &\{(a);(1)\}\quad\text{with $w_{1}=0$},\\
    &\{(2-a);(1)\}\quad\text{with $w_{2}=x^{2}+t$}
\end{align*}
at the irregular singular point $x=\infty$. And we can see that 
\[
\delta(P^{tc})=\Phi^{-1}(\mathbf{m}(P^{tc}))=c_{1}+c_{2}
\]
is the imaginary root of the root system with the Dynkin diagram,

\hspace{40mm}
\begin{xy}
    \ar@2{-}(0,0) *++!D{c_{1}} *\cir<4pt>{}; 
    (10,0) *++!D{c_{2}} *\cir<4pt>{}
\end{xy}.

\vspace{3mm}
\noindent (5) Doubly confluent Heun differential operator.

The doubly confluent Heun differential operator is 
\[
P^{dc}=x^{2}\partial^{2}+(-x^{2}+cx+t)\partial+(-ax+\lambda).
\]

The spectra are
\begin{align*}
    &\{(0);(1)\}\quad\text{with $w^{0}_{1}=0$},\\
    &\{(2-c);(1)\}\quad\text{with $w^{(1)}_{1}=\frac{-t}{x}$},
\end{align*}
at $x=0$ and 
\begin{align*}
    &\{(a);(1)\}\quad\text{with $w^{\infty}_{1}=0$},\\
    &\{(c-a);(1)\}\quad\text{with $w^{\infty}_{2}=x$}
\end{align*}
at $x=\infty$. 
Then the corresponding diagram is 

\hspace{40mm}
\begin{xy}
    \ar@{=} (0,0) *++!D{c_{1}} *\cir<4pt>{}; (10,0) *++!D{c_{2}} *\cir<4pt>{},
    \ar@{};(20,0) *{\bigoplus},
    \ar@{=}(30,0) *++!D{c_{3}} *\cir<4pt>{}; (40,0) *++!D{c_{4}} *\cir<4pt>{}
\end{xy},
and 
\[
\delta(P^{dc})=\Phi^{-1}(\mathbf{m}(P^{dc}))
=\sum_{i=1}^{4}c_{i}+a(c_{1}+c_{2}-c_{3}-c_{4})\ (a\in \Z).
\]
Here we notice that 
$(c_{1}+c_{2}-c_{3}-c_{4})\in \mathrm{Ker\,}\Phi$. 
We can see that $\delta(P^{dc})$ is $W(P^{dc})$-invariant. 
Hence we can conclude that $P^{dc}$ has Weyl group $W(P^{dc})$ symmetry.

Let us give comments about the relationship with Painlev\'e equations. 
As is known, if we put an apparent singular point 
to each Heun operators and consider the isomonodromic deformation, 
then we can obtain  Painlev\'e equations, 
namely, $P_{VI}$ from the Heun operator, 
$P_{V}$ from the confluent Heun operator, $P_{VI}$ 
from the biconfluent Heun operator, 
$P_{III}$ from the doubly confluent Heun  operator, 
and $P_{II}$ from the triconfluent Heun operator respectively
(see \cite{Ok}).

It is known that these Painlev\'e equations have following 
affine Weyl group symmetries generated by B\"acklund transformations.
\[
 \begin{array}{|c|c|c|c|c|}
     \hline
     P_{VI}&P_{V}&P_{IV}&P_{II}&P_{III}\\
     \hline
     D^{(1)}_{4}&A^{(1)}_{3}&A^{(1)}_{2}&A^{(1)}_{1}&(A_{1}\oplus A_{1})^{(1)}\\
     \hline
 \end{array}
\]
Our Weyl groups recover these Painlev\'e symmetries. 
\subsection{$\Phi$-root system}
  We shall define an analogue of the root system in $L(P)$, 
  called $\Phi$-roots and show that if $P$ is irreducible, 
  then the spectral type $\left(m^{[i,j]}_{1},\ldots,
  m^{[i,j]}_{s_{[i,j]}}\right)
  _{\substack{0\le i\le p\\1\le j\le r_{i}}}\in L(P)$ is a $\Phi$-root.
  
  \begin{df}[Minimal element]
      \normalfont
      If $\mathbf{a}\in L(P)^{+}\backslash\{0\}$
      satisfies the following conditions, 
      we say that $\mathbf{a}$ is {\em minimal}.
      For any $\mathbf{j}\in \mathcal{J}(P)$, 
      we have $\mathrm{rank\,}\sigma(\mathbf{j})(\mathbf{a})
      \ge \mathrm{rank\,}\mathbf{a}$
      or $\sigma(\mathbf{j})(\mathbf{a})\notin L(P)^{+}$.
  \end{df}
  \begin{df}\label{basicgeneric}
      \normalfont
      Let us consider $(\mathbf{a},\overline{\alpha})\in S(P)$ 
      with $\mathbf{a}\in L(P)^{+}\backslash\{0\}$.
      Then 
      we say that $(\mathbf{a},\overline{\alpha})$ is {\em generic} 
      if the following are satisfied.

      There exist a minimal element $\mathbf{m}$ and $w\in \tilde{W}(P)$
      such that $w(\mathbf{a})=\mathbf{m}$ and $w$ is decomposed
      as 
      $w=\sigma_{q}\sigma({\mathbf{j}_{q}})
      \sigma_{q-1}\sigma(\mathbf{j}_{q-1})\cdots
      \sigma(\mathbf{j}_{1})\sigma_{0}$.
      Here $\sigma_{l}\in \langle\sigma(i,j,k)\mid 0\le i\le p,\,
      1\le j\le r_{i},\,1\le k\le s_{[i,j]}-1\rangle$ 
      and $\sigma(\mathbf{j}_{l})$, 
      are chosen as follows.
      \begin{enumerate}
          \item For $l=1,\ldots,q$, 
              $\sigma_{l}\sigma(\mathbf{j}_{l})\cdots
      \sigma(\mathbf{j}_{1})\sigma_{0}(\mathbf{a})\in L^{+}(P)$.
  \item The condition $(\ref{generic})$ in Proposition 
      \ref{Twisted Euler transform}
      is valid for  
      $\sigma_{0}(\overline{\alpha})$ and 
      all $\sigma_{l}\sigma(\mathbf{j}_{l})\cdots 
      \sigma_{1}\sigma(\mathbf{j}_{1})\sigma_{0}(\overline{\alpha})$,
      $l=1,\ldots,q$.
  \end{enumerate}
  \end{df}
Let us define $\Phi$-root system of $L(P)$ 
as an analogue of the root system of $Q(P)$. 
First recall the definition of roots of $Q(P)$. 
\textit{Real roots} are elements in
\[
\Delta_{\text{re}}=\bigcup_{c\in \mathcal{C}}W(P)c,
\]
the union of $W(P)$-orbits of $c\in\mathcal{C}$. 
To define imaginary roots, let us consider the set  
\[
F=\{\alpha\in Q(P)^{+}
=\sum_{c\in \mathcal{C}}\Z_{\ge 0}c\mid 
\substack{\langle \alpha,c\rangle \le 0\text{ for all }c
\in \mathcal{C},\\ \supp(\alpha)
\text{ is connected.}}\}\backslash\{\mathbf{0}\}.
\]
Here we say $\supp(\alpha)$ is connected 
if $\alpha=\sum_{c\in\mathcal{C}}\alpha_{c}c$ satisfies the following. 
If $I=\{c\in \mathcal{C}\mid \alpha_{c}\neq0\}$ is decomposed 
by a disjoint union $I=I_{1}\amalg I_{2}$ 
such that we have $\langle c_{1},c_{2}\rangle =0$ 
for all $c_{1}\in I_{1}$ and $c_{2}\in I_{2}$, 
then $I_{1}=\emptyset$ or $I_{2}=\emptyset$.

Then \textit{imaginary roots} are elements in
\[
\Delta_{\text{im}}=W(P)F\cup -(W(P)F).
\]

Also \textit{roots} are elements in 
\[
\Delta=\Delta_{\text{re}}\cup \Delta_{\text{im}}.
\]

Let us define $\Phi$-roots as an analogue of $\Delta$. 
We define \textit{$\Phi$-real roots} as elements in
\[
\Delta^{\Phi}_{\text{re}}
=\bigcup_{\mathbf{j}\in\mathcal{J}(P)}\tilde{W}(P)\Phi(c_{\mathbf{j}}).
\]
Define the subset
\[
F^{\Phi}
=
\left\{\mathbf{a}=\left(a^{[i,j]}_{1},\ldots,
a^{[i,j]}_{s_{[i,j]}}\right)
_{\substack{0\le i\le p\\1\le j\le r_{i}}}\in L(P)^{+}
\backslash\{\mathbf{0}\}\,\Big|\,
\substack{a^{[i,j]}_{1}\ge a^{[i,j]}_{2}\ge \cdots 
\ge a^{[i,j]}_{s_{[i,j]}},\, d(\mathbf{a};\mathbf{j})\ge 0\\
\text{for all }i=0,\ldots,p,\,j=1,\ldots,r_{i},
\mathbf{j}\in\mathcal{J}(P)}\right\}.
\]
Then \textit{$\Phi$-imaginary roots} are elements in
\[
\Delta^{\Phi}_{\text{im}}=\tilde{W}(P)F^{\Phi}\cup -(\tilde{W}(P)F^{\Phi}).
\]
Also \textit{$\Phi$-roots} are elements in 
$\Delta^{\Phi}=\Delta^{\Phi}_{\text{re}}\cup \Delta^{\Phi}_{\text{im}} 
$
and in particular \textit{$\Phi$-positive roots} are elements in 
$
\Delta^{\Phi+}=\Delta^{\Phi}\cap  \prod_{i=0}^{p}
\prod_{j=1}^{k_{i}}\Z_{\ge 0}^{l_{i,j}}
$.

The next proposition shows  that $\Delta^{\Phi}$ can be seen as a
generalization of the root system  $\Delta$.
\begin{prop}
    For any $\mathbf{a}\in \Delta^{\Phi}$, 
    there exists $\alpha\in \Delta$ such that 
    $\Phi(\alpha)=\mathbf{a}$.
\end{prop}
\begin{proof}
    It is clear that 
    $\Phi(\Delta^{\text{re}})\supset \Delta^{\Phi}_{\text{re}}$
    by definition. 
    Also we have  $\Phi(F)\supset F^{\Phi}$ from Lemma 6 in 
    \cite{HO}. 
    Thus $\Phi(\Delta^{\text{im}})\supset \Delta^{\Phi}_{\text{im}}$.
\end{proof}
The next theorem shows that if  the differential operator $P$ is 
irreducible in $W(x)$,
then the corresponding spectral type 
$\left(m^{[i,j]_{1}},\ldots,
m^{[i,j]}_{s_{[i,j]}}\right)
_{\substack{0\le i\le p\\1\le j\le r_{i}}}\in L(P)$
is a $\Phi$-root.
\begin{thm}\label{root and irreducible}
    Let us consider $P\in W[x]$ as above. 
    Suppose that $P$ is irreducible in $W(x)$ and 
    $(\mathbf{m},\overline{\lambda})
    =\left( (m^{[i,j]}_{1},\ldots,m^{[i,j]}_{s_{[i,j]}})
    _{\substack{0\le i\le p\\1\le j\le r_{i}}},\,
    (\lambda^{[i,j]}_{1},\ldots,\lambda^{[i,j]}_{s_{[i,j]}})
    _{\substack{0\le i\le p\\1\le j\le r_{i}}}\right)\in S(P)$
    is generic.
    Then we have the following. 
    \begin{enumerate}
        \item The spectral type  $\mathbf{m}\in L(P)$ is in $\Delta^{\Phi+}$. 
        \item If $\mathrm{idx\,}\mathbf{m}>0$, 
            then $\mathrm{idx\,}\mathbf{m}=2$. 
        \item We have
    \[
    \mathbf{m}\in\begin{cases}
        \Delta^{\Phi}_{\text{re}}&\text{if }\mathrm{idx\,}\mathbf{m}=2,\\
        \Delta^{\Phi}_{\text{im}}&\text{if }\mathrm{idx\,}\mathbf{m}\le 0.
    \end{cases}
    \]
\end{enumerate}
\end{thm}
\begin{proof}
    Since $(\mathbf{m},\overline{\lambda})$ is generic, we can choose 
    minimal element $\mathbf{a}\in L(P)^{+}$ and $w\in \tilde{W}(P)$
    as in Definition \ref{basicgeneric}.
    If there exist $\mathbf{j}_{0}\in \mathcal{J}(P)$ such that 
    $\sigma(\mathbf{j}_{0})(\mathbf{a})\notin L(P)^{+}$, 
    then Proposition \ref{irreducibility is preserved} shows 
    $\mathrm{rank\,}\mathbf{a}=1$. 
    Thus there exist $\sigma\in \langle \sigma(i,j,k)\mid
    0\le i\le p,\,1\le j\le r_{i},\,1\le k\le s_{[i,j]}\rangle$
    and $\mathbf{j}\in \mathcal{J}(P)$ 
    such that $\sigma(\mathbf{a})\in\Phi(c_{\mathbf{j}})$.
    Hence $\mathbf{m}\in \Delta^{\Phi}_{\text{re}}$ and 
    $\mathrm{idx\,}\mathbf{m}=\mathrm{idx\,}\mathbf{a}=2$.
    
    Next we assume 
    $\mathrm{rank\,}\sigma(\mathbf{j})(\mathbf{a})
    \ge \mathrm{rank\,}\mathbf{a}$
    for any $\mathbf{j}\in \mathcal{J}(P)$ and 
    we show that  $\mathrm{idx\,}\mathbf{m}\le 0$. 
    Applying elements in $\langle \sigma(i,j,k)\mid 0\le i\le p,\,
    1\le j\le r_{i},\,1\le k\le s_{[i,j]}-1\rangle$ to $\mathbf{a}$,
    we may assume $\mathbf{a}=\left(a^{[i,j]}_{1},\ldots,
    a^{[i,j]}_{s_{[i,j]}}\right)
    _{\substack{0\le i\le p\\1\le j\le r_{i}}}$ satisfies that
    $a^{[i,j]}_{1}\ge a^{[i,j]}_{2}\ge \cdots\ge a^{[i,j]}_{s_{[i,j]}}$
    for all $i=0,\ldots,p$ and $j=1,\ldots,r_{i}$.
    Then Lemma 6 in \cite{HO} shows that there exists $\alpha\in F$ 
    such that $\Phi(\alpha)=\mathbf{a}$. Thus 
    $\mathrm{idx\,}\mathbf{a}=\langle \alpha,\alpha\rangle \le 0$.
\end{proof}
In the theory of the middle convolution (cf. \cite{K}), 
the Katz algorithm is one of the most important results, 
which shows that if an irreducible Fuchsian differential operator 
or a local system is rigid, i.e., uniquely determined by 
local structures around their singular points, 
i.e., equivalent classes of local monodromies at singular points,
then this operator or local system can be reduced to 
rank 1 element by finite iteration of 
the middle convolutions and the additions. 
This rigidity condition is estimated by the certain number, 
so-called the index of rigidity. 
Namely, one can show that a Fuchsian differential operator 
or local system are rigid if and only 
if their index of rigidity is 2.

A generalization of this theorem 
for non-Fuchsian differential operators 
is obtained by D. Arinkin and D. Yamakawa independently 
(see \cite{A} and \cite{Y}). 
We can show an analogue of their results 
as a immediate consequence of Theorem \ref{root and irreducible}.
\begin{cor}[Cf. Arinkin \cite{A} and Yamakawa \cite{Y}]
    We use the same notation as in Theorem \ref{root and irreducible}.
    Suppose that $P$ is irreducible in $W(x)$ and 
    $(\mathbf{m},\overline{\lambda})\in S(P)$ is generic.
    Then we can reduce $P$ to a rank 1 operator 
    by finite iteration of the twisted Euler transform 
    if and only if $\mathrm{idx\,}\mathbf{m}=2.$
\end{cor}

\end{document}